\documentclass{amsart}
\usepackage{graphicx}
\usepackage{amsmath}
\usepackage{mathtools}
\DeclareMathOperator*{\argmax}{argmax}

\newtheorem{theorem}{Theorem}[section]
\newtheorem{lemma}[theorem]{Lemma}
\newtheorem{cor}[theorem]{Corollary}

\theoremstyle{definition}
\newtheorem{definition}[theorem]{Definition}
\newtheorem{example}[theorem]{Example}

\theoremstyle{remark}
\newtheorem{remark}[theorem]{Remark}

\numberwithin{equation}{section}

\begin{document}

\title[Parametric Interpolation and Non-convex Flux Functions]{Parametric Interpolation Framework for 1-D Scalar Conservation Laws with Non-convex Flux Functions}

\author{Geoffrey McGregor}
\address{Department of Mathematics, McGill University, Montreal, QC, Canada.}
\curraddr{805 Sherbrooke St W, Montreal, QC H3A 0B9, Canada.}
\email{Geoffrey.McGregor@mail.mcgill.ca}
\thanks{The research of GMc was supported in part by the Schulich Scholarship at McGill University and the Murata Family Fellowship.}

\author{Jean-Christophe Nave}
\address{Department of Mathematics, McGill University, Montreal, QC, Canada.}
\curraddr{805 Sherbrooke St W, Montreal, QC H3A 0B9, Canada.}
\email{jcnave@math.mcgill.ca}
\thanks{The research of JCN was supported in part by the NSERC Canada Discovery Grants Program. }


\date{\today}



\begin{abstract}
In this paper we present a novel framework for obtaining high order numerical methods for 1-D scalar conservation laws with non-convex flux functions. When solving Riemann problems, the Oleinik entropy condition, \cite{Oleinik}, is satisfied when the resulting shocks and rarefactions correspond to correct portions of the appropriate (upper or lower) convex envelope of the flux function. We show that the standard equal-area principle fails to select these solutions in general, and therefore we introduce a generalized equal-area principle which always selects the weak solution corresponding to the correct convex envelope.  The resulting numerical scheme presented here relies on the area-preserving parametric interpolation framework introduced in \cite{mcgregor2019area} and locates shock position to fifth order in space, conserves area exactly and admits weak solutions which satisfy the Oleinik entropy condition numerically regardless of the initial states. 
\end{abstract}

\maketitle

\section{Introduction}\label{Intro}
In this paper we consider the 1-D scalar conservation law,\\
\begin{equation}
\begin{cases}
u_t+(F(u))_x=0\label{Cauchy}\\
u(x,0)=g(x),
\end{cases}
\end{equation}
where $F$ is smooth and $g$ is piecewise smooth in their respective domains. In particular this article focuses on the case where the flux function $F$ is non-convex. 

The goal in the present work is to introduce a novel framework for numerically solving (\ref{Cauchy}) to high order while preserving exact conservation and ensuring discontinuities satisfy the Oleinik entropy condition of \cite{Oleinik}. In particular the framework presented in this paper captures shock position and rarefaction waves to high order through the use of a generalized equal-area principle and a parametric representation of the characteristic curves associated with (\ref{Cauchy}).

Removing the convexity requirement on the flux function in (\ref{Cauchy}) adds considerable complexity to weak solutions. In particular, when solving a Riemann problem in the convex case,  the result is either a rarefaction wave, or a shock wave. For more on the convex case, see \cite{Lax,Evans,LeVeque}. In the non-convex case however, Riemann problems may result in a sequence of rarefaction and shock waves linked together. The resulting sequence of shocks and rarefactions is solely determined by the initial states and the flux function $F$. As discussed in \cite{LevFinite,LeFloch}, one way of constructing weak solutions which satisfy the required admissibility condition as given by Oleinik in \cite{Oleinik} is through the convex envelope of $F$. In \cite{LeFloch}, the authors propose a numerical method for constructing weak solutions of (\ref{Cauchy}) for both the convex and non-convex cases. Their Convex Hull Algorithm, denoted CHA, utilizes the method of characteristics along with a convex hull transformation to construct the corresponding shock and rarefaction profiles. The open-source software utilized in \cite{LeFloch} to construct the convex hull within the Convex Hull Algorithm can be found through a link in their paper. In \cite{Seibold}, the authors present an exactly conservative numerical method also relying on the method of characteristics. Although \cite{Seibold} mostly treats the convex case, the method indeed can be applied to the non-convex case with a few tweaks to the algorithm. Numerical simulations showed that their method achieved first order accuracy in the non-convex case. The authors conjectured that the reduction of one order from the convex case was a result of the inflection point in the flux function.

Similar to the work presented in \cite{LeFloch,Seibold}, we seek a method which relies on the characteristic equations associated with (\ref{Cauchy}), given by
\begin{align}
\dot{x}&=F'(u) \label{CharEq}\\
\dot{u}&=0, \nonumber
\end{align}
which can be solved exactly,
\begin{align}
x(x_0,t)&=x_0+F'(g(x_0))t\nonumber\\
u(x_0,t)&=g(x_0),\label{HomParCurve}
\end{align}
where $g(x)$ is the given initial condition of (\ref{Cauchy}). Written as a planar curve parametrized by $x_0$, the solution to (\ref{CharEq}) is given by

\begin{align}
\langle x(x_0,t) \, ,\, u(x_0) \rangle=\langle x_0+F'(g(x_0))t \, , \, g(x_0) \rangle.\label{ParCurve}
\end{align}
The curve $\langle x(x_0,t) \, ,\, u(x_0) \rangle$ remains a parametrization of the strong solution to (\ref{Cauchy}) up to some time $t^*$, provided $\frac{\partial}{\partial x_0}x(x_0,t)>0$ for all $x_0$ in the computation domain and $0\leq t<t^*$. If at some point $x_0^*$ we have $\frac{\partial}{\partial x_0}x(x_0^*,t)=1+F''(g(x_0^*))g'(x_0^*)t<0$, for some $t>0$, then the parametric curve becomes multi-valued and a projection is required to recover the appropriate weak solution to (\ref{Cauchy}). In the convex case, as discussed in \cite{LevFinite} and \cite{mcgregor2019parametric}, the correct weak solution can be obtained through an equal-area projection (also known as the equal-area principle), see Figure \ref{Proj Sketch} for an illustration. The vertical line in Figure \ref{Proj Sketch} will move at the correct shock speed given by the Rankine-Hugoniot condition  \cite{Rankine,Hugo}, provided the regions on either side of the vertical line have the same area. For a proof of the equal-area principle in the convex case see \cite{mcgregor2019parametric}. One objective of the present work is to investigate the link between the equal-area principle and the convex hull of the flux function $F$. Similar to the convex case, an appropriate application of the equal-area principle requires a high order and precise representation of the characteristic curve obtained by solving (\ref{CharEq}). To achieve this we turn to the parametric interpolation literature.
\begin{figure}[!ht]
\begin{center}
\includegraphics[width=40mm,height=30mm]{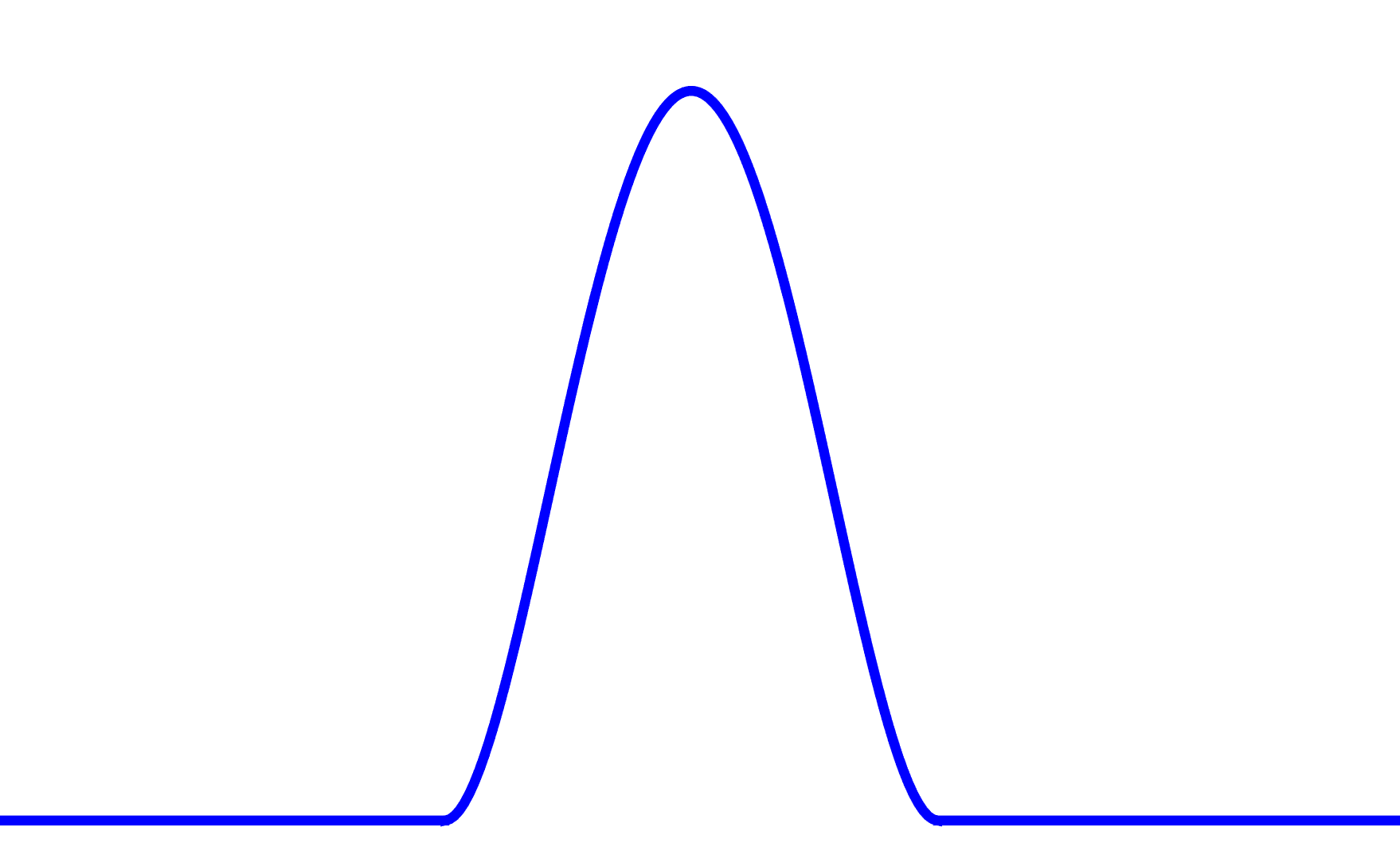}
\includegraphics[width=40mm,height=30mm]{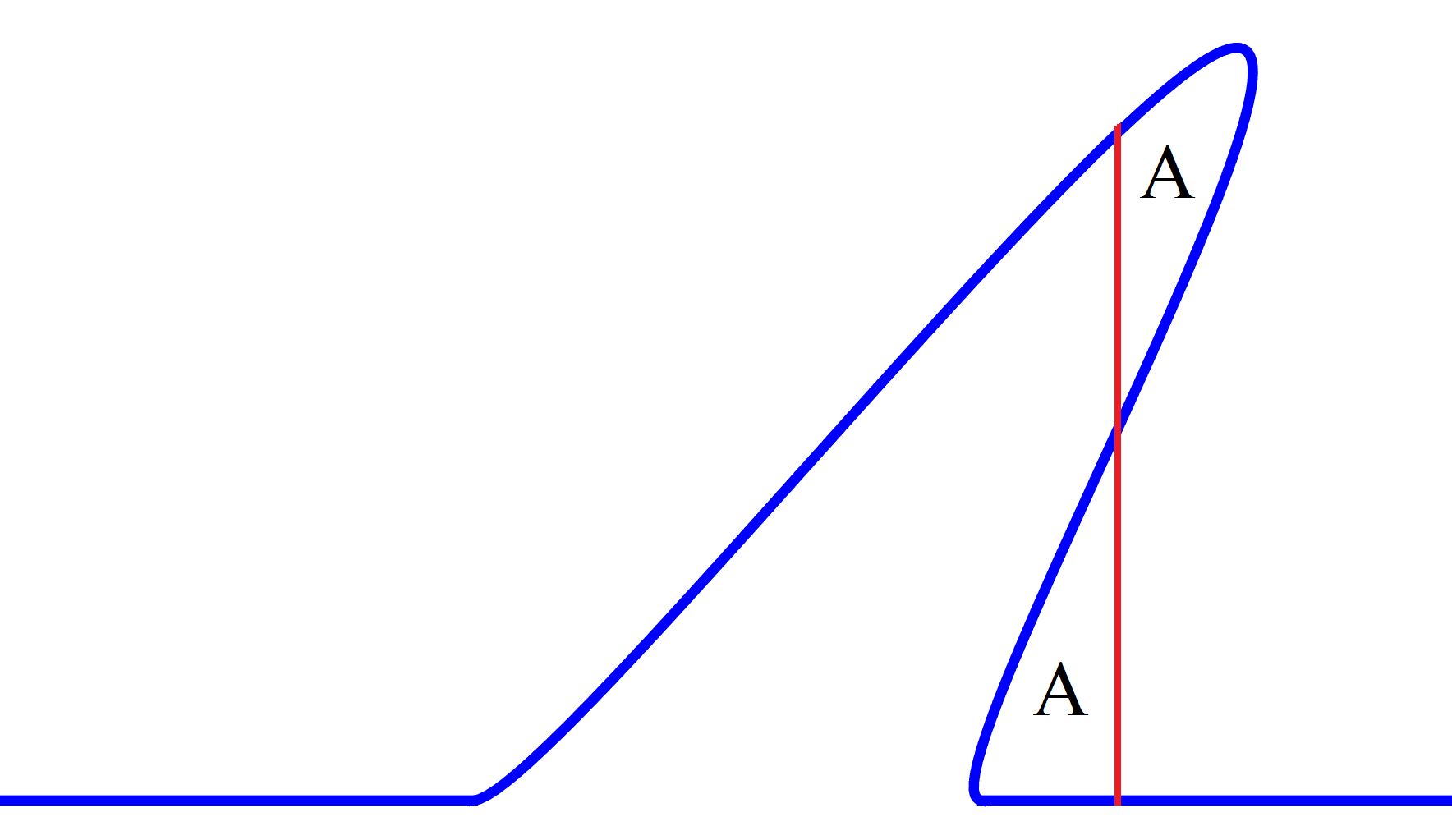}
\includegraphics[width=40mm,height=30mm]{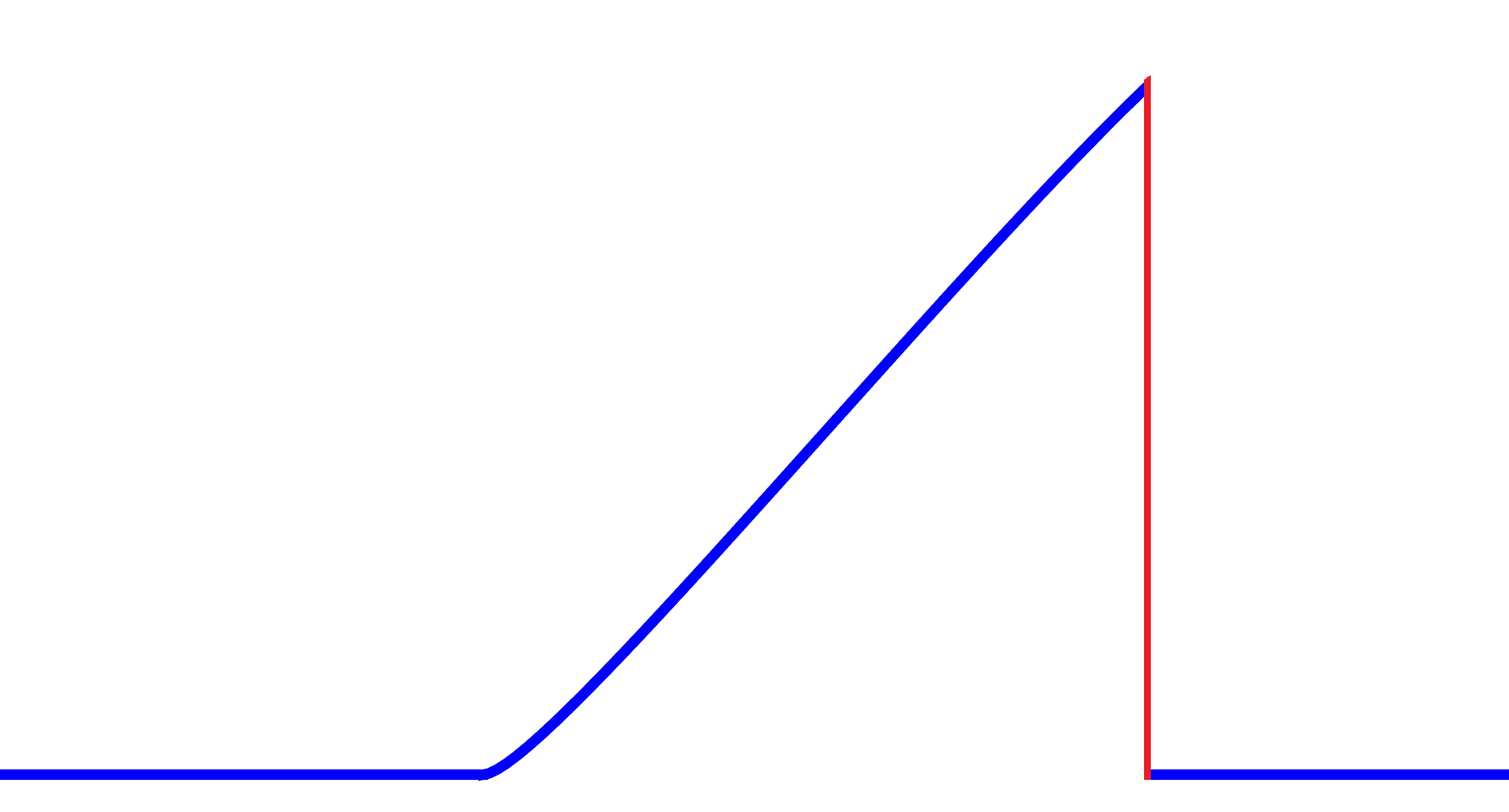}
\end{center}
\caption{ An illustration of the equal-area projection in the convex case. }
\label{Proj Sketch}
\end{figure} 

The literature on parametric interpolation is extensive, however it is rarely utilized as a tool in numerical methods for differential equations. Given that we are concerned with high accuracy and smoothness of the solution curve, as opposed to the smoothness of the parametrization itself, we seek interpolation methods with high geometric continuity. Geometric continuity was first introduced in \cite{DeBoore}, where the authors matched function value, tangent direction and curvature at endpoints to obtain up to sixth order accuracy. Numerous other interpolation techniques can be employed to achieved desired characteristics, such as matching prescribed arc length, see \cite{ArcLength,farouki1}, or minimizing the curvature variation energy in \cite{CurvMin}, or the strain energy in \cite{StrainEnergy}.  Given that we are interested in obtaining high order numerical schemes, we focus on parametric interpolation methods emphasizing accuracy.

The conservative nature of (\ref{Cauchy}) is fundamental to the resulting solution, for example, conserving area, as in the equal-area principle, is sufficient to select the appropriate weak solution in the convex flux function case. With this in mind, we are interested in applying the area-preserving parametric interpolation method discussed in \cite{mcgregor2019area}. In that work the authors construct a family of exactly area-preserving parametric Hermite polynomials which are fifth order accurate, one order higher than the standard parametric cubic Hermite. 

This paper is organized as follows: First, in Section \ref{AreaInterpolation} we present a brief overview of cubic B\'ezier interpolation and discuss the area-preserving framework of \cite{mcgregor2019area}. Next, in Section \ref{PIF}, we show how the parametric interpolation framework can be applied to homogeneous scalar conservation laws in one-space dimension. In particular we show that through the use of the Generalized Equal-Area Principle we are able to construct weak solutions which correspond to the appropriate convex envelope of $F$. Finally in Section \ref{Numerics} we show some numerical results and conclude with a discussion in Section \ref{Discussion}.

\section{Area-Preserving Parametric Interpolation}\label{AreaInterpolation}
In this section we present a brief overview of the area-preserving B\'ezier interpolation discussed in \cite{mcgregor2019area}. The main objective that work may be summarized as follows: given a planar parametric curve $\langle f(s), g(s) \rangle$, parametrized by $s\in[s_0,s_1]$, find a cubic B\'ezier polynomial defined by
\begin{equation}
\vec{B}(t)=\langle B_1(t) \, , \, B_2(t) \rangle=\vec{A}(1-t)^3+3\vec{C}_1(1-t)^2t+3\vec{C_2}(1-t)t^2+\vec{D}t^3, \quad \text{for $t\in[0,1]$},\label{Bezier}
\end{equation}
which satisfies
\begin{align*}
&\vec{B}(0)=\langle f(s_0) \, , \, g(s_0)\rangle, \quad \vec{B}(1)=\langle f(s_1) \, , \, g(s_1)\rangle,\nonumber\\
&\vec{B}'(0)=r_1\langle f'(s_0) \, , \, g'(s_0)\rangle=r_1\vec{\alpha},\quad \text{for some } r_2\in\mathbb{R}\nonumber\\
&\vec{B}'(1)=r_2\langle f'(s_1) \, , \, g'(s_1)\rangle=r_2\vec{\beta},\quad \text{for some } r_2\in\mathbb{R} , \text{and finally}\nonumber\\
&\int_0^1B_2(\tau)B_1'(\tau)\text{d}\tau=\int_{s_0}^{s_1}g(\tau)f'(\tau)\text{d}\tau. \quad \text{(Equal-area constraint)}
\end{align*}
The coefficients $\vec{A},\vec{C}_1,\vec{C}_2$ and $\vec{D}$ are extracted from the functions $f$ and $g$ above, however, an additional degree of freedom remains.  After translating the data to the origin ($\vec{A}=0$), the Equal-area constraint leads to a relation in terms of $r_1$ and $r_2$, given by
\begin{equation}
\frac{r_1r_2}{60}(\vec{\alpha}\times\vec{\beta} )+\frac{r_1}{10}(\vec{D}\times\vec{\alpha})+\frac{r_2}{10}(\vec{\beta}\times\vec{D})+\frac{D_1D_2}{2}=\int_{s_0}^{s_1}g(\tau)f'(\tau)\text{d}\tau-g(s_0)((f(s_1)-f(s_0)),\label{AreaPres}
\end{equation}
where $\times$ denotes the scalar vector product $\vec{\alpha}\times\vec{\beta}=\alpha_1\beta_2-\beta_1\alpha_2$. The main result in \cite{mcgregor2019area} shows that the interpolation is fifth order accurate in the $L^{\infty}$ norm provided the parameters $r_1$ and $r_2$ satisfy the Equal-area constraint (\ref{AreaPres}) and an appropriate decay rate ( as $||\vec{D}-\vec{A}||\rightarrow 0 $).  It is important to note that this result assumes the portion of the parametric curve being interpolated is small enough such that it may be represented by some function $\langle x,\tilde{f}(x)\rangle$ after an appropriate rotation.  The simplest fifth order area-preserving cubic B\'ezier interpolation of the function $\langle x,\tilde{f}(x)\rangle$ over the interval $x\in[0,h]$ is given by taking $r_1=h$ and solving for $r_2$ in (\ref{AreaPres}). This is a natural choice since the resulting curve can be viewed as a perturbed cubic Hermite polynomial. For more details see \cite{mcgregor2019area}.

In the forthcoming section we provide details on how to apply this parametric interpolation framework to homogeneous scalar conservation laws in one space dimension.

\section{Parametric Interpolation Framework for 1-D Homogeneous Scalar Conservation laws}\label{PIF}
The purpose of this section is to discuss the properties of the parametric curve (\ref{ParCurve}), $\langle x_0+F'(g(x_0))t, g(x_0) \rangle$, given by solving the characteristic equations (\ref{CharEq}). In particular we are interested in the relationship between area-preserving projections of (\ref{ParCurve}) and the convex envelope of the flux function $F$ in (\ref{Cauchy}). 
We begin by showing that the data required to construct a parametric polynomial interpolant, as in Section \ref{AreaInterpolation}, is easily obtained from the parametric curve (\ref{ParCurve}).

Suppose we wish to interpolate (\ref{ParCurve}) between the parametrization values of $s_0$ and $s_1$ at time $\tau$. As discussed in Section \ref{AreaInterpolation}, to utilize the area-preserving interpolation of \cite{mcgregor2019area} we require value and tangent direction data at each endpoint, along with the parametric area under the curve between $s_0$ and $s_1$. Therefore, constructing a cubic B\'ezier interpolant, $\vec{B}(t)$ for $t\in[0,1]$, uses the data given by
\begin{align*}
&\vec{B}(0)=\langle s_0+F'(g(s_0))\tau , g(s_0)\rangle, \quad \vec{B}(1)=\langle s_1+F'(g(s_1))\tau , g(s_1)\rangle,\nonumber\\
&\vec{B}'(0)=\langle 1+F''(g(s_0))g'(s_0)\tau , g'(s_0)\rangle,\\
&\vec{B}'(1)=\langle 1+F''(g(s_1))g'(s_1)\tau , g'(s_1)\rangle ,\quad  \text{and finally}\nonumber\\
&\int_0^1B_2(s)B_1'(s)\text{d}s=\int_{s_0}^{s_1}g(s)\left(1+F''(g(s))g'(s)\tau\right)\text{d}s.
\end{align*}
We first notice that the endpoint values and the tangent direction at $s_0$ and $s_1$ are easily obtained through simple evaluation of $F$ and $g$, along with their respective derivatives. We note that if the piecewise smooth initial condition, $g$, is not differentiable at an endpoint, then we must use the one sided limit from the interior of the interpolation domain to extract this data. The final piece of data which yields the parametric area under (\ref{ParCurve}) between $s_0$ and $s_1$ can be significantly simplified through a substitution and integration by parts, yielding
\begin{align}
\int_{s_0}^{s_1}g(s)&\left(1+F''(g(s))g'(s)\tau\right)\text{d}s=\int_{s_0}^{s_1}g(s)\text{d}s+\tau\left(F'(g(s))g(s)-F(g(s))\right)\biggr |_{s=s_0}^{s_1}\nonumber\\
=&\int_{s_0}^{s_1}g(s)\text{d}s+\tau\left(F'(g(s_1))g(s_1)-F'(g(s_0))g(s_0)+F(g(s_0))-F(g(s_1))\right).\label{ParArea}
\end{align}
We see from equation (\ref{ParArea}) that updating a parametric interpolant for a later time step, say $\tau+\Delta\tau$, does not require any additional integration. Instead we perform an initial integration at time $\tau=0$ and then are able to update the parametric area through simple evaluation of the second term in (\ref{ParArea}). This is advantageous since we know from \cite{mcgregor2019area} that adding the area constraint yields an additional order of accuracy, and as can be seen from this calculation, it does not create a significant increase in computational cost.

Another essential property of the parametric curve (\ref{ParCurve}) is that area change can only occur over the boundary. For example, taking $g(s_0)=g(s_1)$ and using equation (\ref{ParArea}) we have
\begin{align*}
\frac{d}{dt}\int_{s_0}^{s_1}g(s)\left(1+F''(g(s))g'(s)\tau\right)\text{d}s&=F'(g(s_1))g(s_1)-F'(g(s_0))g(s_0)+F(g(s_1))-F(g(s_0))\\
&=F'(g(s_1))g(s_1)-F'(g(s_1))g(s_1)+F(g(s_1))-F(g(s_1))\\
&=0.
\end{align*}
This is the main motivation behind the equal-area principle seen in Figure \ref{Proj Sketch}. We know that the area under weak solutions of (\ref{Cauchy}) can only change via flux over the boundary, and since the parametric curve (\ref{ParCurve}) preserves this same structure, then the weak solution must come from an area-preserving projection of the multi-valued portion of the curve. The subtly here is that the overturned region obtained from (\ref{ParCurve}) where the flux function $F$ is non-convex may be far more complicated than in convex examples. For instance, see Figure \ref{NonConRiemann}.

\begin{figure}[!ht]
\begin{center}
\includegraphics[width=40mm,height=30mm]{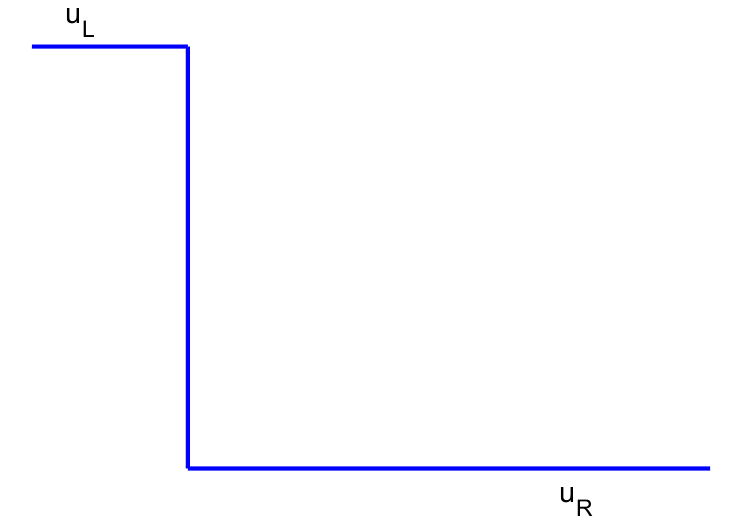}
\includegraphics[width=40mm,height=30mm]{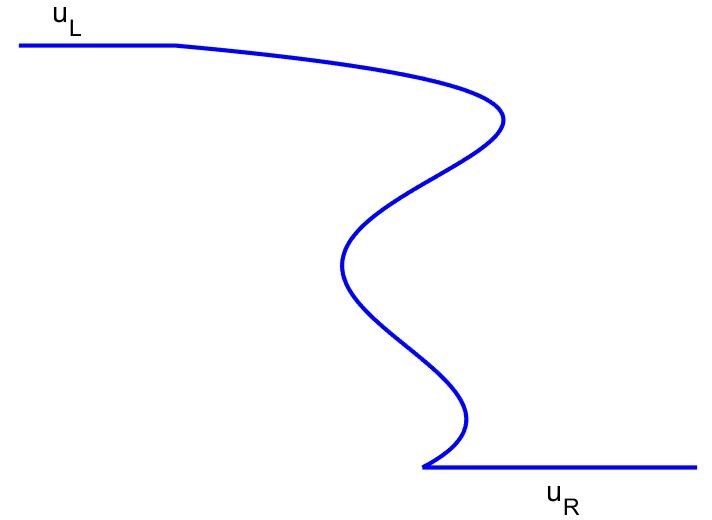}
\end{center}
\caption{Example of a Riemann problem for a non-convex flux function. }
\label{NonConRiemann}
\end{figure}

The reason the example shown in Figure \ref{NonConRiemann} creates an additional difficulty is due to the overturned region not yielding a unique equal-area solution. For example, Figure \ref{NonConRiemann2} shows two possible equal-area solutions corresponding to the flux function shown in Figure \ref{RiemannFlux1}. In both cases the area to the right of the vertical line is the same as the area to the left, therefore replacing these overturned curves with the vertical shock lines preserves area. 

\begin{figure}[!ht]
\begin{center}
\includegraphics[width=40mm,height=30mm]{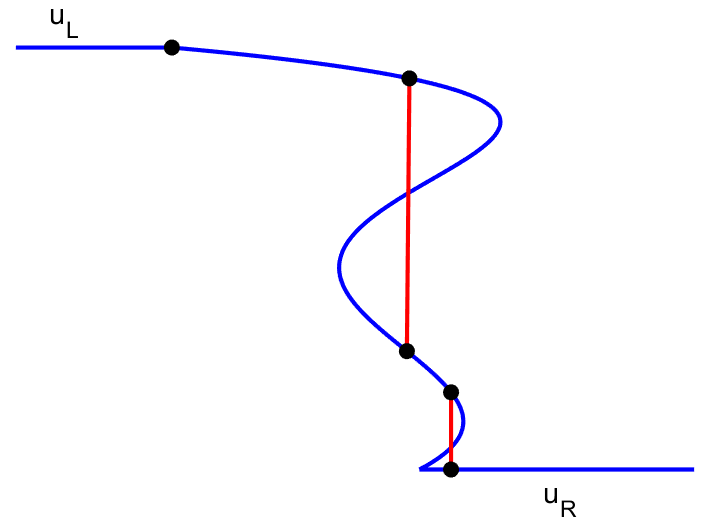}
\includegraphics[width=40mm,height=30mm]{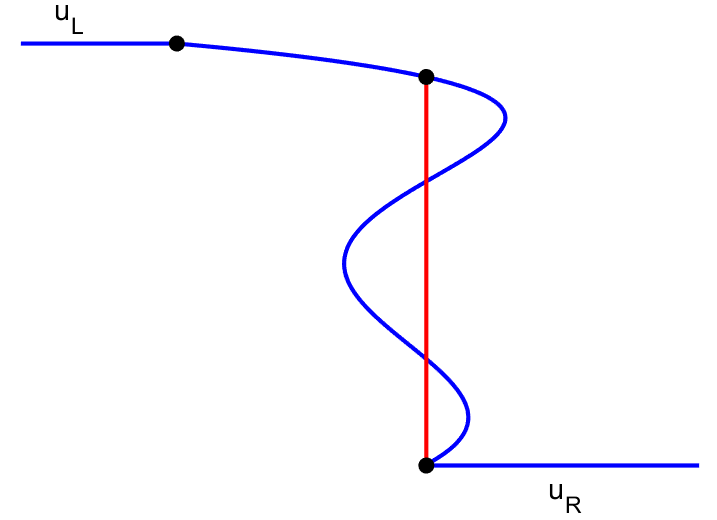}
\end{center}
\caption{Example where the equal-area principle fails to yield a unique solution. }
\label{NonConRiemann2}
\end{figure}

The remaining part of this section is dedicated to understanding equal-area solutions corresponding to the parametric curve (\ref{ParCurve}) and then show that we can always find an equal-area solution which corresponds exactly to the appropriate convex envelope of the flux function $F$ between $u_L$ and $u_R$. We begin by recalling some general properties of convex envelopes for single-variable functions.

\begin{figure}[!ht]
\begin{center}
\includegraphics[width=60mm,height=40mm]{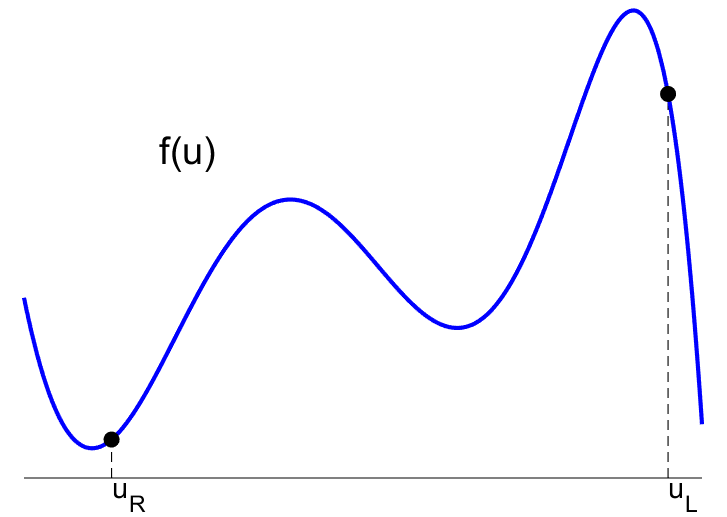}
\end{center}
\caption{Flux function associated with Riemann problem in Figure \ref{NonConRiemann}. }
\label{RiemannFlux1}
\end{figure}

The upper convex hull of a single variable function $F(x)$ between two points $x_0<x_1$ is defined by a function $UconF(x)\geq F(x)$, where $UconF(x)$ is the smallest function such that the set of all $(x,y)\in\mathbb{R}^2$ satisfying $x_0\leq x\leq x_1$ and $y\leq UconF(x)$ is a convex set. Similarly the lower convex hull is defined by the largest function $LconF(x)\leq F(x)$ such that the set of all $(x,y)\in\mathbb{R}^2$ satisfying $x_0\leq x\leq x_1$ and $y\geq LconF(x)$ is a convex set. Therefore, if we can find $UconF(x)$ we have the upper convex hull, and similarly with $LconF(x)$ and the lower convex hull. The standard way of defining $UconF$ and $LconF$, shown in Definition \ref{ConHull}, is through affine functions, for example see \cite{ConvexAnalysis}.  
\begin{definition}\label{ConHull}
The upper convex hull of a function $F(x)$ for $x\in[x_0,x_1]$ is given by the set of all $(x^*,y^*)\in\mathbb{R}^2$ satisfying $y^*\leq UconF(x^*)$, where $UconF$ is given by
\begin{align*}
UconF(x^*)=\inf_{(a,b)\in\mathbb{R}^2}ax^*+b \quad \text{such that} \quad ax+b\geq F(x)\, \forall x\in[x_0,x_1]. 
\end{align*}
Similarly, the lower convex hull of a function $F(x)$ for $x\in[x_0,x_1]$ is given by the set of all $(x^*,y^*)\in\mathbb{R}^2$ satisfying $y^*\geq LconF(x^*)$, where $LconF$ is given by
\begin{align*}
LconF(x^*)=\sup_{(a,b)\in\mathbb{R}^2}ax^*+b \quad \text{such that} \quad ax+b\leq F(x)\, \forall x\in[x_0,x_1]. 
\end{align*}
\end{definition}
Therefore, at each $x^*\in[x_0,x_1]$, the value of $UconF(x^*)$, for example, is given through evaluation of the smallest affine function which is greater than or equal to $F(x)$ everywhere within the domain of computation. Definition (\ref{ConHull}) makes it clear that $UconF$ will consist of line segments and sections where $UconF(x)=F(x)$, and similarly with $LconF(x)$. The upper convex hull corresponding to Figure \ref{RiemannFlux1} is shown in Figure \ref{UpperConHullFig}. Focusing on the upper convex hull for now, we present a few properties which help with the construction of $UconF$ arising from Definition (\ref{ConHull}).

\begin{figure}[!ht]
\begin{center}
\includegraphics[width=60mm,height=40mm]{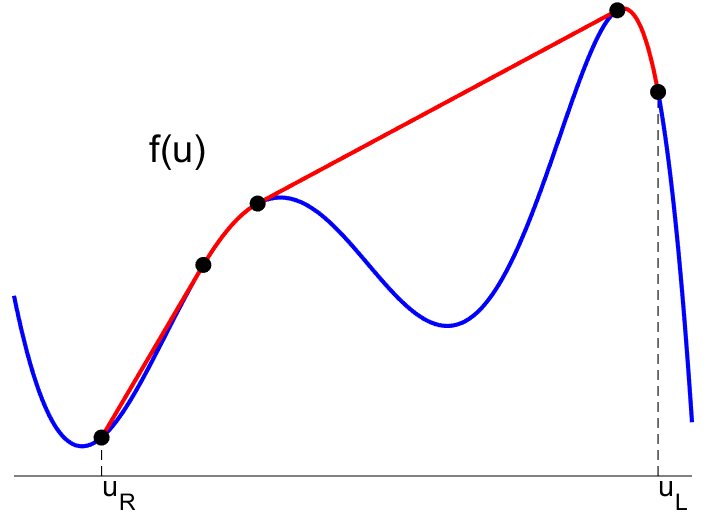}
\end{center}
\caption{Example of upper convex hull of $f(u)$.}
\label{UpperConHullFig}
\end{figure}

First and foremost, for given states $u_R<u_L$ we check if the secant line from $(u_R,F(u_R))$ to $(u_L,F(u_L))$, given by  $y(u)=\frac{F(u_L)-F(u_R)}{u_L-u_R}(u-u_R)+F(u_R)$, satisfies $y(u)\geq F(u)$ for $u\in[u_R,u_L]$. If this is the case, then the $UconF(u)=y(u)$. Therefore from here we assume that the secant $y(u)$ crosses $F(u)$ somewhere between $u_R$ and $u_L$. Supposing that the secant from $u_R$ and $u_L$ fails to be the upper convex hull, one must determine if $UconF(u)=F(u)$ in a neighbourhood of $u_R$, or if $UconF(u)$ connects to $u_R$ through a secant line. We begin by considering the set $T_{u_R}$, defined by
\begin{equation}
T_{u_R}=\left\{ u^*\in(u_R,u_L)\biggr|\,  \frac{F(u^*)-F(u_R)}{u^*-u_R}=F'(u^*)\geq F'(u_R) \right\}.\label{TangentSet}
\end{equation}

Using the above definition of $T_{u_R}$ we have the following Lemma.
\begin{lemma}\label{TangentLemma}
Suppose that the secant line from $u_R$ to $u_L$ fails to be the upper convex hull of the smooth function $F(u)$ in $[u_R,u_L]$. Then, if $T_{u_R}$ is non-empty, then $UconF(u)$ consists of a secant line from $u_R$ to $u^*$ which is tangent to $F(u)$ at $u^*\in T_{u_R}$ where $u^*$ satisfies $F'(u^*)\geq F'(u)$ for all $u\in T_{u_R}$. If the steepest secant is tangent to $F$ in more than one location, then we set $u^*$ to be the largest such point. Otherwise if $T_{u_R}$ is empty, then $UconF(u)=F(u)$ in a neighbourhood to the right of $u_R$.
\end{lemma}
\begin{proof}
Suppose $T_{u_R}$ is non-empty with $u^*\in T_{u_R}$, where $\displaystyle u^*=\argmax_{u\in T_{u_R}}F'(u)$. If the argmax is satisfied at more than one location then set $u^*$ to be the largest value among them. For this secant line to not be part of $UconF(u)$ we would need a smaller function satisfying the convex hull definition (\ref{ConHull}). If $UconF$ contained a secant line from $u_R$ with smaller slope than $F'(u^*)$, then the secant line would fail to satisfy the majorization property of (\ref{ConHull}), as it would eventually intersect $F(u)$ somewhere. The other option is for $UconF(u)=F(u)$ on some interval $[u_R,u_R+\varepsilon]$ for some $\varepsilon>0$. However this requires lines tangent to $F(u)$ within the interval $[u_R,u_R+\varepsilon]$  majorize $F(u)$, and since $u^*\in T_{u_R}$ we know that the secant line from $u_R$ to $u^*$ has slope greater than $F'(u_R)$, meaning that the tangent line at $u_R$ cannot possibly satisfy the majorization criteria of $UconF$ from Definition (\ref{ConHull}).

Now suppose $T_{u_R}$ is empty with the secant line from $u_R$ to $u_L$ failing to be $UconF$. This implies all secant lines from $u_R$ to $u^*$ with $\frac{F(u^*)-F(u_R)}{u^*-u_R}=F'(u^*)$ must have $F'(u^*)<F'(u_R)$. In particular, the tangent line at $u_R$ majorizes $F(u)$ for all $u\in(u_R,u_L]$, otherwise $UconF$ would contain the secant directly to $u_L$, or $T_{u_R}$ would be non-empty. Now consider the function $S(c,u)=F'(c)(u-c)+F(c)$. First we observe that $S(u_R,u)$ is exactly the tangent line from $u_R$, which implies $S(u_R,u)>F(u)$ for all $u\in(u_R,u_L]$. Since we have assumed $F$ is smooth, we have continuity in $S$ with respect to its first variable, which implies that there exists an $\varepsilon>0$ such that $S(u_R+\varepsilon,u)>F(u)$ for all $u\in(u_R,u_L]$. This implies $UconF(u)=F(u)$ in a neighbourhood of $u_R$, which is the smallest possible value for $UconF(u)$ since we require $UconF(u)\geq F(u)$ within $[u_R,u_L]$.
\end{proof}

As mentioned in Lemma \ref{TangentLemma}, when $T_{u_R}$ is empty and the secant line from $u_R$ to $u_L$ fails to be the upper convex hull, we know the tangent line at $u_R$, given by $y_{u_R}(u)=F'(u_R)(u-u_R)+F(u_R)>F(u)$ for $u\in(u_R,u_L]$. This allowed us to prove that $UconF(u)=F(u)$ in some neighbourhood $[u_R,u_R+\varepsilon)$. In the case that the conditions for Lemma \ref{TangentLemma} are satisfied with $T_{u_R}$ non-empty, we have that $UconF(u)$ has a secant line connecting $u_R$ to some $u^*\in(u_R,u_L)$. Using the same argument presented in Lemma \ref{TangentLemma}, the tangent line at $u^*$ majorizes $F(u)$ on $(u^*,u_L]$ and therefore $UconF(u)=F(u)$ in a neighbourhood $[u^*,u^*+\varepsilon)$, for some $\varepsilon>0$.  From this position we can construct the next branch of the upper convex hull.

\begin{lemma}\label{Branch2}
Suppose $\hat{u}\in[u_R,u_L)$ satisfies $UconF(\hat{u})=F(\hat{u})$ with $y_{\hat{u}}(u)=F'(\hat{u})(u-\hat{u})+F(\hat{u})>F(u)$ for $u\in(\hat{u},u_L]$. Then we have $UconF(u)=F(u)$ for $u\in[\hat{u},u^{**}]$, where $u^{**}$ is the smallest $u$ such that $y_{u^{**}}(u_L)=F'(u^{**})(u_L-u^{**})+F(u^{**})=F(u_L)$, or $y_{u^{**}}(\bar{u})=F(\bar{u})$, with $F'(u^{**})=F'(\bar{u})$, for some $\bar{u}\in(u^{**},u_L)$. 
\end{lemma}
\begin{proof}
Suppose $UconF(\hat{u})=F(\hat{u})$ and the tangent line at $\hat{u}$, given by 
$y_{\hat{u}}(u)=F'(\hat{u})(u-\hat{u})+F(\hat{u})$ satisfies $y_{\hat{u}}(u)>F(u)$ for $u\in(\hat{u},u_L]$. Additionally we suppose that $u^{**}>\hat{u}$ is the smallest value such that $y_{u^{**}}(u_L)=F(u_L)$ or $y_{u^{**}}(\bar{u})=F(\bar{u})$ with $F'(u^{**})=F'(\bar{u})$ for some $\bar{u}\in(u^{**},u_L)$. We show that this assumption results in $F''(u)\leq0$ for $u\in [\hat{u},u^{**}]$, which implies $F(u)$ lives below its tangent lines between $\hat{u}$ and $u^{**}$ and therefore $UconF(u)=F(u)$ in $[\hat{u},u^{**}]$.

Suppose that $F''(v)>0$ at some point $v\in(\hat{u},u^{**})$. By continuity of $F''(u)$, there exists some $\varepsilon>0$ such that $F''(u)>0$ for all $u\in (v-\varepsilon,v+\varepsilon)$. This implies that the tangent line at $v$, given by $y_{v}(u)=F'(v)(u-v)+F(v)$, satisfies $y_{v}(v+\varepsilon)<F(v+\varepsilon)$, since $F$ is concave up within $(v-\varepsilon,v+\varepsilon)$. Recalling the function $S(c,u)=F'(c)(u-c)+F(c)$ used in the proof of Lemma \ref{TangentLemma}, we have that $S(v+\varepsilon,u)=y_{v}(v+\varepsilon).$ By assumption that the tangent line at $\hat{u}$ is greater than $F(u)$ we have $S(\hat{u},v+\varepsilon)>F(u+\varepsilon)$. Since $S(u,v+\varepsilon)$ is continuous in its first variable, and $S(\hat{u},v+\varepsilon)>F(u+\varepsilon)$ and $S(v,v+\varepsilon)<F(u+\varepsilon)$, by the intermediate value theorem we have that there exists some $v^*\in(\hat{u},v+\varepsilon)$ such that $S(v^*,v+\varepsilon)=F(v+\varepsilon)$, which contradicts the assumption that $u^{**}$ is the earliest point satisfying $y_{u^{**}}(\bar{u})=F(\bar{u})$ for some $\bar{u}\in(u^{**},u_L)$. The property that $F'(u^{**})=F'(\bar{u})$ is simply a result of all points to the left of $u^{**}$ having tangent lines which are larger than $F$ everywhere.  
\end{proof}

The process described in Lemma \ref{Branch2} is repeated until $u_L$ is reached and function $UconF(u)$ is constructed.

The combination of Lemma \ref{TangentLemma} and Lemma \ref{Branch2} give us a way to construct the upper convex hull of $F$ directly. The following Lemmas show that utilizing the equal-area principle in a particular way is equivalent to the construction described in the previous Lemmas. But first we require the following definition.

\begin{figure}[!ht]
\begin{center}
\includegraphics[width=60mm,height=40mm]{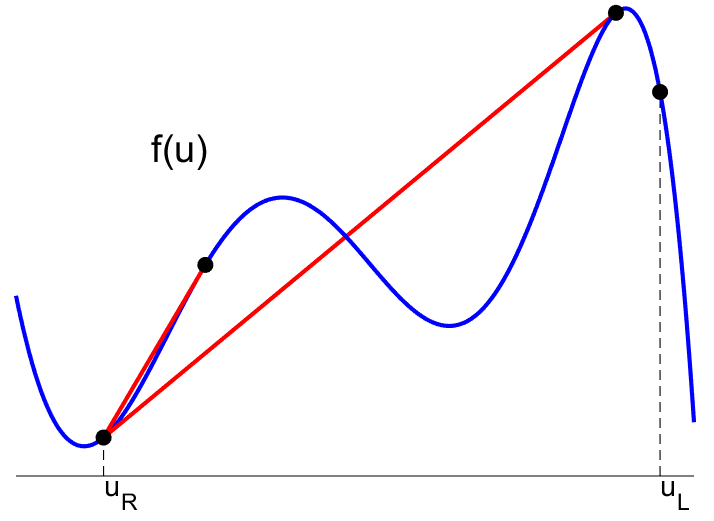}
\end{center}
\caption{Elements of $T_{u_R}$ for the flux function given in Figure \ref{RiemannFlux1}. }
\label{RiemannTangents1}
\end{figure}

\begin{definition}\label{EqualAreaCurve}
We say $\langle x(s), u(s) \rangle$ is an equal-area curve from $s_0$ to $s_1$ provided $x(s_0)=x(s_1)$ and
\begin{equation*}
\int_{s_0}^{s_1}u(s)x'(s)\text{d}s=0.
\end{equation*}
\end{definition}

The following Lemma states that secant lines from $u_R$ which are tangent to $F$ and some $u^*\in(u_R,u_L]$ are equivalent to the existence of an equal-area curve from $u_R$ to $u^*$. 

\begin{lemma}\label{EqualAreaTangent}
There exists a line segments through the point $(u_R,F(u_R))$ which is tangent to the point $(u^*,F(u^*))$ with $u_R<u^*<u_L$ and $F'(u_R)\leq F'(u^*)$ if and only if the parametric curve (\ref{ParCurve}) corresponding to the Riemann problem initially at $x=x_0$ with upper state $u_L$ and lower state $u_R$ has an equal-area line located at position $x=x_0+F'(u^*)\tau$ at time $t=\tau$ connecting $u^*$ to $u_R$. 
\end{lemma}
\begin{proof}
Suppose there is a line segment passing through the point $(u_R,F(u_R))$ which is tangent to the point $(u^*,F(u^*))$ which satisfies $F'(u_R)\leq F'(u^*)$. This implies that the tangent line $y_{u_R}(u)$ satisfies
\begin{align}
y_{u_R}(u)&=F'(u^*)(u^*-u_R)+F(u_R),\quad \text{with}\label{TangentEqn1}\\
y_{u_R}(u^*)&=F(u^*)\label{TangentEqn2}.
\end{align}
Combining these two equations yields
\begin{equation}
F'(u^*)(u^*-u_R)+F(u_R)=F(u^*).\label{TangentCon}
\end{equation}
We now turn our attention to the Riemann problem. Parametrizing the vertical line of the initial condition between $u^*$ and $u_R$ we have 
$\langle x_0,u_Rs+(1-s)u^*\rangle$ for $s\in[0,1]$. Applying the characteristic flow (\ref{CharEq}) for $\tau$ units of time yields the parametric curve $\langle x_0+F'(u_Rs+(1-s)u^*)\tau,u_Rs+(1-s)u^*\rangle$, for $s\in[0,1]$. For this curve to be equal-area about a vertical line at position $x=x_0+F'(u^*)\tau$ between $u^*$ and $u_R$ we first require $F'(u_R)\leq F'(u^*)$ to ensure we have a closed curve, which is true by assumption. We also must have that
\begin{equation}
\int_0^1(u_Rs+(1-s)u^*)(F''(u_Rs+(1-s)u^*)(u^*-u_R)\tau\text{d}s+(F'(u^*)-F'(u_R))u_R\tau=0. \label{EqualAreaLine}
\end{equation}
The integral term in (\ref{EqualAreaLine}) is similar to the one in Definition \ref{EqualAreaCurve} but extended to a piecewise defined curve.

Performing the substitution $v(s)=u_Rs+(1-s)u^*$, then integrating by parts, the integral term becomes
 \begin{align}
\int_0^1(u_Rs+(1-s)u^*)(F''(u_Rs+(1-s)u^*)(u^*-u_R)\tau\text{d}s=\tau(F'(v)v-F(v))\biggr|_{v=u^*}^{u_R}\nonumber\\
=(F'(u_R)u_R-F(u_R))\tau-(F'(u^*)u^*-F(u^*))\tau.\label{EqualAreaLine2}
\end{align}
Applying the tangency condition (\ref{TangentCon}) sets $F'(u^*)(u^*-u_R)+F(u_R)=F(u^*)$ and results in
\begin{equation}
\int_0^1(u_Rs+(1-s)u^*)(F''(u_Rs+(1-s)u^*)(u^*-u_R)\tau\text{d}s=(F'(u_R)-F'(u^*))u_R\tau, \label{EqualAreaLine3}
\end{equation}
which implies that equation (\ref{EqualAreaLine}) is satisfied and therefore $\langle x_0+F'(u_Rs+(1-s)u^*)\tau,u_Rs+(1-s)u^*\rangle$ is indeed an equal-area curve about the vertical line at position $x=x_0+F'(u^*)\tau$. 

Conversely, suppose the solution curve (\ref{ParCurve}) applied to the Riemann problem from $u_L$ to $u_R$ initially at $x_0$ has an equal-area curve at time $\tau$ about $x_0+F'(u^*)\tau$ from $u^*$ to $u_R$. The vertical shock line between $u*$ and $u_R$ maps after $\tau$ units of time to the parametric curve $\langle x_0 + F'(u_Rs+(1-s)u^*)\tau, u_Rs+(1-s)u^* \rangle$. Similarly the bottom shock state, initially parametrized by $\langle x_0+s,u_R\rangle$, for $s\geq 0$, after $\tau $ units of time maps to the curve  $\langle x_0+s+F'(u_R)\tau,u_R\rangle$. To ensure the equal-area curve about $x_0+F'(u^*)\tau$ is indeed a closed curve, we require $x_0+s+F'(u_R)\tau=x_0+F'(u^*)\tau$ for some $s\geq 0$.  Solving for $s$ we obtain $s=F'(u^*)-F'(u_R)$. Therefore, our assumption that a closed curve exists implies $F'(u_R)\leq F'(u^*)$. Our equal area assumption implies that
\begin{equation*}
\int_0^1(u_Rs+(1-s)u^*)(F''(u_Rs+(1-s)u^*)(u^*-u_R)\tau\text{d}s+(F'(u^*)-F'(u_R))u_R\tau=0
\end{equation*}
Computing the integral as done above leads to the equation
\begin{equation*}
\left((F'(u_R)u_R-F(u_R))-(F'(u^*)u^*-F(u^*))\right)\tau+(F'(u^*)-F'(u_R))u_R\tau=0.
\end{equation*}
Simplifying leads us to
\begin{equation*}
F'(u^*)(u^*-u_R)+F(u_R)=F(u^*),
\end{equation*}
which is exactly the tangency condition. Combining this with the result that $F'(u_R)\leq F'(u^*)$ completes the proof.
\end{proof}
The same is true for secant lines from $u_L$ which are tangent to some $u^{**}$ with $F'(u^{**})\leq F'(u_L)$. We present the result in the following Corollary but we omit the proof as it mirrors that of Lemma \ref{EqualAreaTangent} exactly.
\begin{cor}\label{Corollary}
There exists a line segments through the point $(u_L,F(u_L))$ which is tangent to the point $(u^{**},F(u^{**}))$ with $u_R<u^{**}<u_L$ and $F'(u^{**})\leq F'(u_L)$ if and only if the parametric curve (\ref{ParCurve}) corresponding to the Riemann problem initially at $x=x_0$ with upper state $u_L$ and lower state $u_R$ has an equal-area line located at position $x=x_0+F'(u^{**})\tau$ at time $t=\tau$ connecting $u^{**}$ to $u_L$.
\end{cor}

Lemma \ref{EqualAreaTangent} tells us that, for a given Riemann problem from $u_L$ to $u_R$, searching the parametric curve (\ref{ParCurve}) for equal-area curves which attach to the state $u_R$ is equivalent to searching for elements $u^*\in T_{u_R}$. As discussed in Lemma \ref{TangentLemma}, provided $T_{u_R}$ is non-empty and the secant from $u_R$ to $u_L$ does not yield $UconF$, the upper convex hull is determined by selecting the element $u^* \in T_{u_R}$ with maximum slope $F'(u^*)$. Therefore, in this situation, the equal-area curve from some $u^*\in (u_R,u_L)$ to $u_R$ which corresponds to the upper convex hull maximizes $F'(u^*)$. Since the shock location at time $\tau$ is given by $x_0+F'(u^*)\tau$, it is equivalent to seeking the equal-area curve from $u^*$ to $u_R$ with the largest shock location. 
The next situation we examine is when $UconF(u)$ consists of a secant line from $u_R$ to $u_L$. This implies that $F'(u_R)\leq \frac{F(u_L)-F(u_R)}{u_L-u_R}\leq F'(u_L)$ and $\frac{F(u_L)-F(u_R)}{u_L-u_R}\geq F'(u^*)$ for all $u^*\in T_{u_R}$. The following Lemma states that $UconF(u)$ consists of a secant line from $u_R$ to $u_L$ if and only if there is an equal-area curve from $u_L$ to $u_R$ corresponding to a shock position which exceeds all other shocks which connect to $u_R$.

\begin{lemma}\label{LemmaBigShock}
Assume a Riemann problem at $x=x_0$ with left state $u_L$ and right state $u_R$ satisfying $u_R<u_L$. Then, $UconF(u)$ consists of a secant line from $u_R$ to $u_L$ if and only if there is an equal-area curve from $u_L$ to $u_R$ corresponding to a shock position of $x_0+\frac{F(u_L)-F(u_R)}{u_L-u_R}\tau$ at time $\tau$ which exceeds all other shocks which connect to $u_R$.
\end{lemma}
\begin{proof}
Consider the Riemann problem at $x=x_0$ from $u_L$ to $u_R$ with $u_L>u_R$. Suppose that $UconF(u)$ consists of a secant line from $u_R$ to $u_L$. This implies that $F'(u_R)\leq \frac{F(u_L)-F(u_R)}{u_L-u_R}\leq F'(u_L)$ and $\frac{F(u_L)-F(u_R)}{u_L-u_R}\geq F'(u^*)$ for all $u^*\in T_{u_R}$. 

To have an equal-area curve from $u_L$ to $u_R$ about a shock at position $x^*(\tau)=x_0+S^*(\tau)$ we require
\begin{equation}
\left(F'(u_L)\tau-S^*(\tau)\right)u_L+\left(S^*(\tau)-F'(u_R)\tau\right)u_R+\tau\int_{u_L}^{u_R}vF''(v)\text{d}v=0,\label{BigEqualAreaShock}
\end{equation}
along with $F'(u_R)\tau\leq S^*(\tau)\leq F'(u_L)\tau$. Simplifying (\ref{BigEqualAreaShock}) tells us that if $S^*(\tau)=\frac{F(u_L)-F(u_R)}{u_L-u_R}\tau$ then indeed the equation is satisfied. To satisfy Definition \ref{EqualAreaCurve} we require that the vertical line at $x^*(\tau)$ crosses both the left and right states of the shock, given by  $\langle x_0-r+F'(u_L)\tau,u_L\rangle$ for $r\geq 0$ and $\langle x_0+s+F'(u_R)\tau,u_R\rangle$ for $s\geq 0$ respectively . This requires the existence an $r\geq 0$ such that $x_0-r+F'(u_L)\tau=x_0+\frac{F(u_L)-F(u_R)}{u_L-u_R}\tau$, which yields $r=\left(F'(u_L)-\frac{F(u_L)-F(u_R)}{u_L-u_R}\right)\tau$, which is indeed positive by assumption. Similarly with the right state we require the existence of an $s\geq 0$ satisfying $x_0+s+F'(u_R)\tau=x_0+\frac{F(u_L)-F(u_R)}{u_L-u_R}\tau$, implying $s=\left(\frac{F(u_L)-F(u_R)}{u_L-u_R}-F'(u_R)\right)\tau$, also positive by assumption. Combining these results implies the existence of an equal-area curve from $u_L$ to $u_R$. The final requirement for the proof is that the resulting shock at the location $x^*(\tau)=x_0+\frac{F(u_L)-F(u_R)}{u_L-u_R}\tau$ exceeds all other shocks generated by equal-area curves to $u_R$. Since $UconF(u)$ consists of a secant line from $u_R$ to $u_L$, this implies that $\frac{F(u_L)-F(u_R)}{u_L-u_R}\geq F'(u^*)$ for all $u^*\in T_{u_R}$ and therefore $x_0+\frac{F(u_L)-F(u_R)}{u_L-u_R}\tau \geq x_0+F'(u^*)\tau$ for all $u^*\in T_{u_R}$ as well. By Lemma \ref{EqualAreaTangent} we know that each $u^*\in T_{u_R}$ corresponds to an equal-area curve about $x_0+F'(u^*)\tau$, which by the above argument must satisfy $x_0+F'(u^*)\tau \leq x_0 + \frac{F(u_L)-F(u_R)}{u_L-u_R}\tau$. Therefore the equal-area curve corresponding to the maximum shock position is given by the equal-area curve from $u_L$ to $u_R$.
\\
\\
Conversely, Suppose there is an equal-area curve from $u_L$ to $u_R$ which corresponds to a shock with position greater than or equal to any other equal-area curve to $u_R$. This implies equation (\ref{BigEqualAreaShock}) is satisfied with $S^*(\tau)=\frac{F(u_L)-F(u_R)}{u_L-u_R}\tau$ and $\frac{F(u_L)-F(u_R)}{u_L-u_R}\geq F'(u^*)$ for all $u^*\in T_{u_R}$ and $F'(u_R)\leq \frac{F(u_L)-F(u_R)}{u_L-u_R} \leq F'(u_L)$. We show that this is sufficient to prove that $UconF(u)=\frac{F(u_L)-F(u_R)}{u_L-u_R}(u-u_R)+F(u_R).$ 

Suppose with the above assumptions that there exists a point $\hat{u}$ with $F(\hat{u})>\frac{F(u_L)-F(u_R)}{u_L-u_R}(\hat{u}-u_R)+F(u_R)$, which would prove that $UconF(u)$ fails to yield the upper convex hull. We consider the functions $G(u)=F(u)-\left(\frac{F(u_L)-F(u_R)}{u_L-u_R}(u-u_R)+F(u_R)\right)$ and $D_{G}(u)=G'(u)(u-u_R)+G(u_R)=G'(u)(u-u_R)$ since $G(u_R)=0$. $G(u)$ attains its maximum over the interval $(u_R,u_L)$ at some point $v\in(u_R,u_L)$ which implies $G(v)>0$ since $G(\hat{u})>0$ by assumption. We note that since $G'(v)=0$ we have $D_G(v)=0$. We now let $w\in[u_R,v)$ such that $G(w)=0$ with $G(u)>0$ for all $u\in (w,v)$. Therefore $w$ is the root of $G$ such that $G(u)>0$ between $w$ and its maximum $v$. We know this always exists since $w=u_R$ is a possible choice if $G(u)>0$ for $u \in (u_R,v)$. Next we define $\displaystyle w^*=\argmax_{u\in[w,v]}G'(u)$.

Using the Lagrange remainder theorem for the first order Taylor expansion of $G$ about $w$, we obtain $G(u)=G(w)+G'(z)(u-w)=G'(z)(u-w)$ since $G(w)=0$, for some $z\in(w,u)$. This implies $G(w^*)=G'(z)(w^*-w)$ for some $z\in(w,w^*)$. Since $w\geq u_R$ and $G'(z)\leq G'(w^*)$ we have $G'(z)(w^*-w)\leq G'(w^*)(w^*-u_R)=D_G(w^*)$, implying $G(w^*)\leq D(w^*)$. This guarantees $G(u)=D(u)$ at some $v^*\in [w^*,v)$ by the intermediate value theorem, since $G(w^*)\leq D(w^*)$ and $D(v)<G(v)$. We need that $G(v^*)>0$ in order for the argument to be valid.  If $v^*\in(w^*,v)$ then we have $G(v^*)>0$ since $G(u)>0$ for all $u\in(w,v]$. Next we deal with the case that $D(v^*)=G(v^*)=0$, which implies $v^*=w^*=w$. Provided $w>u_R$ we have $G'(z)(w^*-w)< G'(z)(w^*-u_R)\leq D_G(w^*)$ which yields $G(w^*)<D(w^*)$ which implies $v^*\in (w^*,v)$ and therefore $G(v^*)>0$. This argument fails if $v^*=w^*=w=u_R$, but in this case this implies the maximum of $G'(u)$ between $u_R$ and $v$ occurs at $u_R$. But by assumption $G'(u_R)\leq 0$ which contradicts $G(v)>0$. Putting all of this together we have that there exists a point $v^*\in (u_R,u_L)$ such that $D(v^*)=G(v^*)>0$, which implies $G'(v^*)(v^*-u_R)=G(v^*)$. Using our definition of $G$ we obtain
\begin{align*}
\left(F'(v^*)-\frac{F(u_L)-F(u_R)}{u_L-u_R}\right)(v^*-u_R)&=F(v^*)-\left(\frac{F(u_L)-F(u_R)}{u_L-u_R}(v^*-u_R)+F(u_R)\right)\\
\Rightarrow F'(v^*)&=\frac{F(v^*)-F(u_R)}{v^*-u_R},
\end{align*}
which implies $v^*\in T_{u_R}$ and therefore there exists an equal-area curve from $v^*$ to $u_R$. Finally, since $G(v^*)>0$ we have 
\begin{align*}
\frac{F(v^*)-F(u_R}{v^*-u_R}&>\frac{F(u_L)-F(u_R)}{u_L-u_R},
\end{align*}
which implies $F'(v^*)>\frac{F(u_L)-F(u_R)}{u_L-u_R}$, by the tangency condition. This implies that the equal-area curve connecting $v^*$ has a larger shock position than the shock from $u_L$ to $u_R$, which is a contradiction of our initial assumption.
\end{proof}

Using Lemma \ref{EqualAreaTangent} and Lemma \ref{LemmaBigShock} we know how to apply the equal-area principle appropriately to select the correct shock which connects to $u_R$. As discussed above, when we don't have a secant line connecting to $u_R$ then particles in a neighbourhood of $u_R$ simply travel at their characteristic speed. The same happens to particles in a neighbourhood of $u^*$ which are outside of the shock.  This happens automatically when applying the equal-area approach since all particles travel at their characteristic speed and only the particles within the equal-area curve are replaced by a shock. The following Lemma describes how to proceed from $u_R$ in the absence of a shock and similarly how to proceed from $u^*$ in the presence of a shock from $u^*$ to $u_R$.

\begin{lemma}\label{LemmaIntTangents}
Assume a Riemann problem with left state $u_L$ and right state $u_R$ initially at $x=x_0$. Then there is an equal-area curve from $w\in (u_R,u_L)$ to $w^*\in (w,u_L)$ if and only if the secant line from $w$ to $w^*$ is tangent to $F$ at both $w$ and $w^*$.
\end{lemma}
\begin{proof}
Consider the Riemann problem with left state $u_L$ and right state $u_R$ initially at $x=x_0$. Suppose we have an equal-area curve from $w=u_Rs_0+(1-s_0)u_L$ to $w^*=u_Rs^*+(1-s^*)u_L$. This implies $x(s_0)=x_0+F'(w)\tau=x(s^*)=x_0+F'(w^*)\tau$, which implies that $F'(w)=F'(w^*)$. Additionally we have that $\int_{w}^{w^*}vF''(v)\text{d}v=0,$ which implies
\begin{equation*}
(F'(w^*)w^*-F(w^*))-(F'(w)w-F(w))=0.
\end{equation*}
Using that $F'(w)=F'(w^*)$ we obtain
\begin{align*}
F'(w)(w^*-w)+F(w)&=F(w^*), \quad \text{and}\\
F'(w)=\frac{F(w^*)-F(w)}{w^*-w}&=F'(w^*),
\end{align*}
which combine to imply that the secant line from $w$ to $w^*$ is tangent to $F$ at both $w$ and $w^*$.

Now we suppose that the secant line from $w$ to $w^*$ is tangent to $F$ at both $w$ and $w^*$.  This implies $F'(w)(w^*-w)+F(w)=F(w^*)$ an $F'(w)=F'(w^*)$. Parametrizing the initial shock front with $u_Rs+(1-s)u_L$ we take $w=u_Rs_0+(1-s_0)u_L$ and $w^*=u_Rs^*+(1-s^*)u_L$. Since $F'(w)=F'(w^*)$ we have that $x(s_0)=x(s^*)$. In this case we can compute the area about the vertical line at $x(s_0)$ from $w$ to $w^*$, which is given by
\begin{equation*}
\int_w^{w^*}vF''(v)\text{d}v=(F'(w^*)w^*-F(w^*))-(F'(w)w-F(w)).
\end{equation*}
Applying the tangency conditions yields
\begin{align*}
\int_w^{w^*}vF''(v)\text{d}v&=(F'(w)w^*-(F'(w)(w^*-w)+F(w)))-(F'(w)w-F(w))\\
&=0.
\end{align*} 
Therefore, this gives an equivalence between secant lines within the interior of $UconF$ and equal-area curves within $(u_R,u_L)$.
\end{proof}
 

Putting all of this together we consider the following algorithm.\\

\begin{center}
\textbf{The Generalized Equal-Area Principle}
\end{center}
\begin{enumerate}
\item Parametrize the initial left state by $\langle x_0-r, u_L \rangle$, for $r\geq 0$, the initial right state by $\langle x_0+q, u_R \rangle$, for $q\geq 0$, and the initial wave front by $\langle x_0, u_Rs+(1-s)u_L \rangle$, for $s\in[0,1]$. Then flow the particles under their characteristic flow for $\tau$ units of time. This yields the left state $\langle x_0-r+F'(u_L)\tau,u_L \rangle$, for $r\geq 0$, the right state $\langle x_0+q+F'(u_R)\tau,u_R \rangle$, for $q\geq 0$, and the initial wave front becomes $\langle x_0+F'(u_Rs+(1-s)u_L)\tau,u_Rs+(1-s)u_L \rangle$ for $s\in [0,1]$.
\item Search for equal-area curves from $\langle x_0+F'(u_Rs+(1-s)u_L)\tau,u_Rs+(1-s)u_L \rangle$ for $s^*\in(0,1)$ to the state $u_R$ and also check for an equal-area curve from $u_L$ to $u_R$. If such curves exist then select the one corresponding to largest shock position. If the largest shock position corresponds to an equal-area curve from $u_L$ to $u_R$ then replace $\langle x_0+F'(u_Rs+(1-s)u_L)\tau,u_Rs+(1-s)u_L \rangle$ for $s\in[0,1]$ by the vertical line at position $x_0+\frac{F(u_L)-F(u_R)}{u_L-u_R}\tau$ and change $r\geq 0$ to $r\geq (F'(u_L)-\frac{F(u_L)-F(u_R)}{u_L-u_R})\tau$ and $q$ to $q\geq (\frac{F(u_L)-F(u_R)}{u_L-u_R}-F'(u_R)\tau$ and the algorithm concludes. If the largest shock connects from parameter value $s^*\in(0,1)$, then replace $\langle x_0+F'(u_Rs+(1-s)u_L)\tau,u_Rs+(1-s)u_L \rangle$ for $s\in[s^*,1]$ by the vertical line at position $x_0+F'(u_Rs+(1-s)u_L)\tau$ and the right state becomes defined only for $q\geq (F'(u_Rs^*+(1-s^*)u_L)-F'(u_R))\tau$. If no shock to $u_R$ exists, proceed to the next step with $s^*=1$.
\item Search for the largest $\hat{s}\in (0,s^*]$ such that there is an equal-area curve with end point $\langle x_0+F'(u_R\hat{s}+(1-\hat{s})u_L)\tau, u_R\hat{s}+(1-\hat{s})u_L \rangle$. If no shock exists the algorithm concludes.  If the shock connects to $u_L$ then $\langle x_0+F'(u_Rs+(1-s)u_L)\tau,u_Rs+(1-s)u_L \rangle$ for $s\in[0,\hat{s}]$ is replaced by the vertical line at position $x_0+F'(u_R\hat{s}+(1-\hat{s})u_L)\tau$ and the left state becomes defined for $r\geq (F'(u_L)-F'(u_R\hat{s}+(1-\hat{s})u_L))\tau$ and then the algorithm concludes. Otherwise the shock connects to some $s^{**}\in(0,s^*)$ and $\langle x_0+F'(u_Rs+(1-s)u_L)\tau,u_Rs+(1-s)u_L \rangle$ for $s\in[s^{**},s^*]$ is replaced by the vertical line at position $x_0+F'(u_Rs^*+(1-s^*)u_L)\tau$. In this case redo this step starting from $s^{**}$.
\end{enumerate}

\begin{theorem}\label{TheoremEquivalenceShock}
Assume a Riemann problem at $x=x_0$ between two constants states, $u_L$ on the left and $u_R$ on the right with $u_R<u_L$. Then, the Generalized Equal-Area Principle presented above generates a weak solution which corresponds exactly to $UconF(u)$. In particular, each vertical line of the weak solution between heights $u$ and $u^*$ corresponds to a secant line from $F(u)$ to $F(u^*)$ while the remaining portions of the curve corresponds to intervals where $UconF(u)=F(u)$.
\end{theorem}
\begin{proof}
Here we show that indeed the Generalized Equal-Area Principle constructs the weak solution which corresponds to the upper convex hull of $F$. Step 2 attempts to find equal-area curves which connect to $u_R$. If the equal-area curve connecting to $u_R$ with largest shock location connects to $\langle x_0+F'(u_Rs^*+(1-s^*)u_L)\tau,u_Rs^*+(1-s^*)u_L \rangle$ for some $s^*\in(0,1)$ then, as described in Lemma \ref{EqualAreaTangent}, this implies that $u_Rs^*+(1-s^*)u_L\in T_{u_R}$ and it corresponds to the steepest secant of all members of $T_{u_R}$. By Lemma \ref{TangentLemma}, this implies that $UconF(u)$ contains a secant line from $u_R$ to $u_Rs^*+(1-s^*)u_L$. Similarly, if the equal-area curve corresponding to the largest shock location is from $u_L$ to $u_R$, then by the converse of Lemma \ref{LemmaBigShock} this implies that the secant line from $u_R$ to $u_L$ is exactly $UconF(u)$. We proceed to Step 3 provided there aren't any equal-area curves connecting to $u_R$, or if the largest shock location occurs at $\langle x_0+F'(u_Rs^*+(1-s^*)u_L)\tau,u_Rs^*+(1-s^*)u_L \rangle$ for some $s^*\in(0,1)$. 

Step 3 begins by searching for the largest $\hat{s}\in (0,s^*]$ where $\langle x_0+F'(u_R\hat{s}+(1-\hat{s})u_L)\tau, u_R\hat{s}+(1-\hat{s})u_L \rangle$ is the endpoint of an equal-area curve. This equal-area curve connects to some $s^{**}\in(0,\hat{s})$, or to $u_L$. In the case that the shock connects from $s^{**}$ to $\hat{s}$, we know by Lemma \ref{LemmaIntTangents} that this corresponds to a secant line of $F(u)$ which is tangent at both $u_R\hat{s}+(1-\hat{s})u_L$ and $u_Rs^{**}+(1-s^{**})u_L$. In the case that the shock at $\hat{s}$ connects to $u_L$, then Corollary \ref{Corollary} tells us that this corresponds to a secant line of $F$ from $u_R\hat{s}+(1-\hat{s})u_L$ to $u_L$ which is tangent at $u_R\hat{s}+(1-\hat{s})u_L$. Since $u_Rs+(1-s)u_L$ parametrizes from $u_L$ to $u_R$, finding the largest $\hat{s}$ with this property is equivalent to finding the smallest $\hat{u}\in (u_Rs^*+(1-s^*)u_L , u_L)$ with this property, which is exactly what is described in Lemma \ref{Branch2}, which completes the proof.  
\end{proof}

\begin{remark}\label{Reverse}
The Generalized Equal-Area Principle Algorithm can also be applied in the reverse direction. Starting from $u_L$ and searching for the shock with the smallest location and so on. Depending on how the parametric curve is parametrized this may be more convenient from an algorithmic perspective.  In terms of the flux function, this means that the upper convex hull can be constructed through the same method as described in Lemma \ref{TangentLemma} and Lemma \ref{Branch2} but in reverse. Therefore, instead searching for members of $T_{u_R}$, tangent lines from some $u^*\in(u_R,u_L)$ attaching to $u_L$ satisfying $S\leq F'(u_L)$ are required. If no tangent lines exists then $UconF(u)=F(u)$ in a neighbourhood of $u_L$. Progressing from $u_L$ in the absence of a tangent line, or from $u^*$ corresponding to the shallowest tangent line, towards $u_R$ is done in the same way as described in Lemma \ref{Branch2} except traversing from right to left.
\end{remark}

\begin{theorem}\label{TheoremEquivalenceRarefaction}
Assume a Riemann problem at $x=x_0$ between two constants states, $u_L$ on the left and $u_R$ on the right with $u_L<u_R$. Then, the Generalized Equal-Area Principle algorithm presented above generates a weak solution which corresponds exactly to $LconF(u)$, the lower convex envelope of $F$. In particular, each vertical line of the weak solution between heights $u$ and $u^*$ corresponds to a secant line from $F(u)$ to $F(u^*)$ while the remaining portions of the curve corresponds to intervals where $LconF(u)=F(u)$.
\end{theorem}
\begin{proof} 
The lower convex hull between $u_L$ and $u_R$, with $u_L<u_R$ can be constructed in the same manner as described in Lemma \ref{TangentLemma} and Lemma \ref{Branch2} but instead starting from $u_L$. The first step is to search for tangent lines from some $u^*\in(u_L, u_R)$ connecting to $u_L$ which satisfies $F'(u^*)\leq F'(u_L)$ which does not intersect $F$ anywhere other than $u_L$ and $u^*$ between $u_L$ and $u_R$. If such a tangent line exists then the shallowest such tangent line is the first portion of the lower convex hull. Otherwise $LconF(u)=F(u)$ in a neighbourhood of $u_L$. The process described in Lemma \ref{Branch2} is followed identically, searching for the first $u^{**}$ whose tangent line is tangent to some $\bar{u}$ with $u^{**}<\bar{u}$. We notice that this is exactly the reverse algorithm described in Remark \ref{Reverse}, which means that the generalized equal-area principle algorithm will construct the lower convex hull when $u_L<u_R$. 
\end{proof}

In the following section we present some numerical examples which show that the Generalized Equal-Area Principle indeed constructs the weak solution corresponding to the appropriate convex envelope of the flux function $F$.

\section{Numerical Results}\label{Numerics}

In this section we present several Riemann problems for which the flux function is non-convex. The convergence plots associated with the convex envelopes measure the maximum error among all shock locations, or equivalently, the difference in slopes of the secant lines between the numerical and exact convex envelope. Each example utilizes the fifth order accurate exactly area-preserving parametric interpolation \cite{mcgregor2019area} discussed in Section \ref{AreaInterpolation}. We therefore expect to locate each shock to least fifth order accuracy.

\begin{example}\label{Example1}
\begin{equation}
\begin{cases}
u_t+((u^2-2u)^2)_x=0\label{Example1Equation}\\
u(x,0)=
\begin{cases}
2 \quad x<0\\
0 \quad x\geq 0
\end{cases},
\end{cases}
\end{equation}
In this example $u_R<u_L$, therefore the weak solution corresponds to the upper convex hull of $F$. Figure \ref{Example1Plots} shows that the weak solution contains two shocks separated by a rarefaction. Looking to Figure \ref{Example1Envelope} we see the corresponding upper convex hull and the anticipated fifth order convergence.
\begin{figure}[!ht]
\begin{center}
\includegraphics[width=40mm,height=30mm]{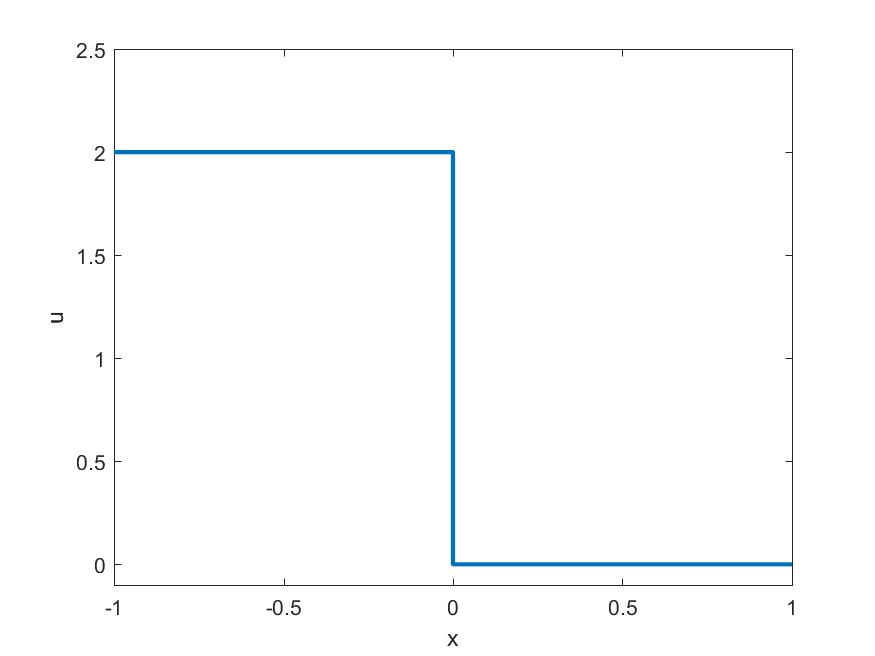}
\includegraphics[width=40mm,height=30mm]{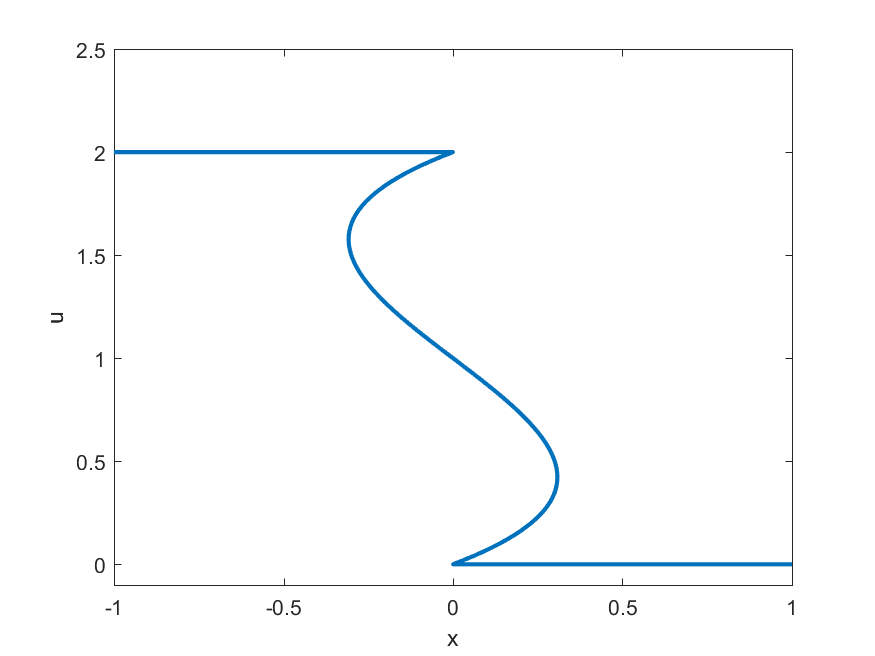}
\includegraphics[width=40mm,height=30mm]{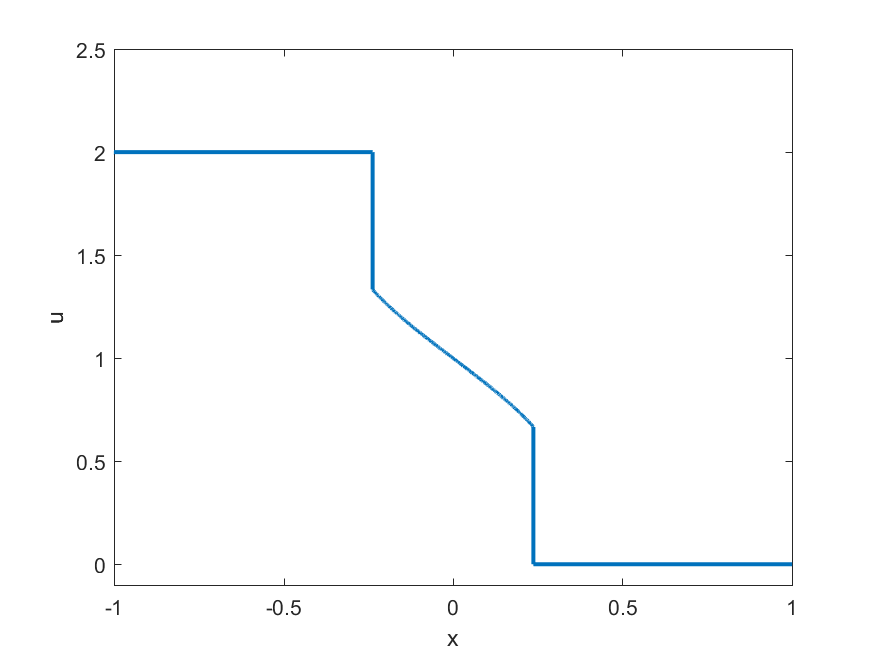}
\end{center}
\caption{Initial condition (left), characteristic flow (centre) and the Generalized Equal-Area Projection (right) associated with Example \ref{Example1}.}
\label{Example1Plots}
\end{figure}

\begin{figure}[!ht]
\begin{center}
\includegraphics[width=60mm,height=40mm]{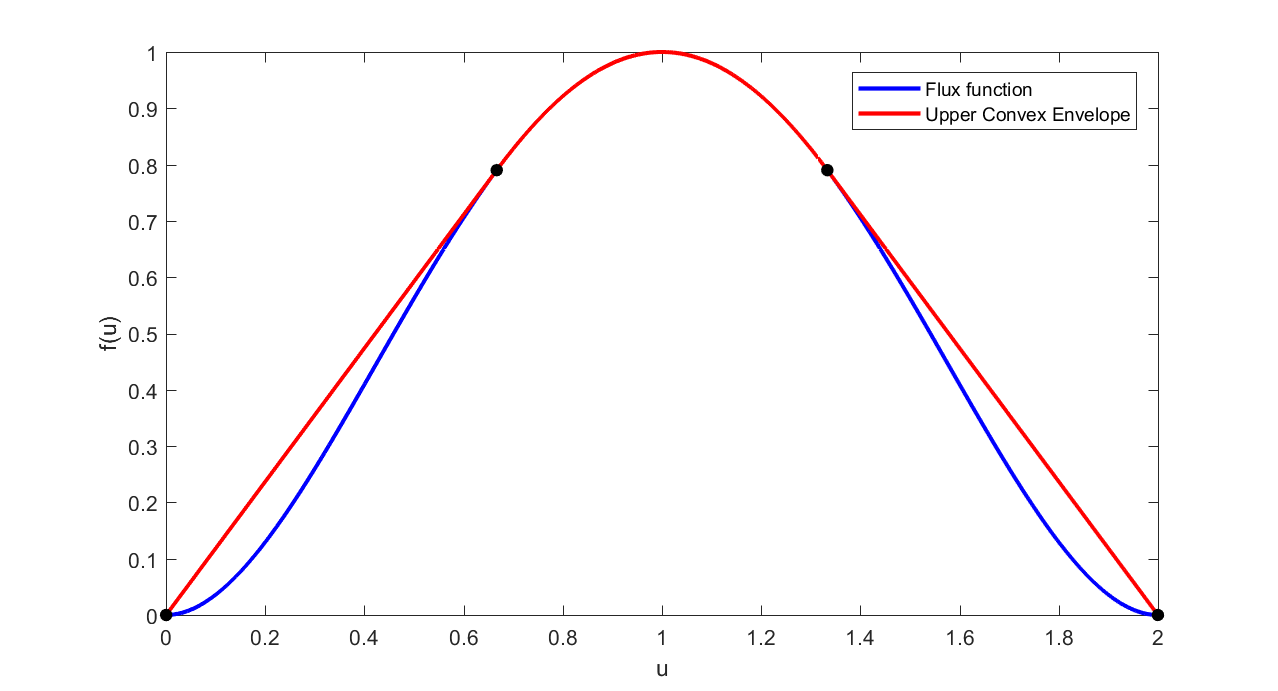}
\includegraphics[width=60mm,height=40mm]{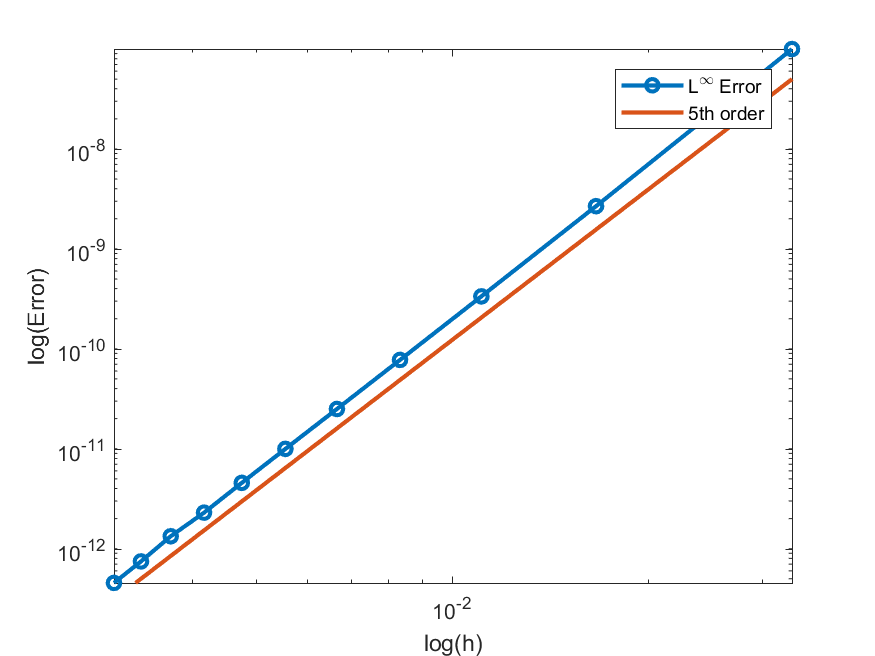}
\end{center}
\caption{Convex envelope (left) and convergence (right) corresponding to Example \ref{Example1}.}
\label{Example1Envelope}
\end{figure}

\end{example}

The next example is the same as Example \ref{Example1} except that the initial states are reversed.
\begin{example}\label{Example2}
\begin{equation}
\begin{cases}
u_t+((u^2-2u)^2)_x=0\label{Example2Equation}\\
u(x,0)=
\begin{cases}
0 \quad x<0\\
2 \quad x\geq 0
\end{cases},
\end{cases}
\end{equation}
In this case we expect to capture the weak solution corresponding to the lower convex hull, since $u_R>u_L$. Indeed, as shown in Figure \ref{Example2Envelope} we capture exactly this, with the resulting weak solution being a standing wave at $x=0$ as shown in Figure \ref{Example2Plots}. We omit the convergence plot in this example as the shock position is located exactly for all partitions.

\begin{figure}[!ht]
\begin{center}
\includegraphics[width=40mm,height=30mm]{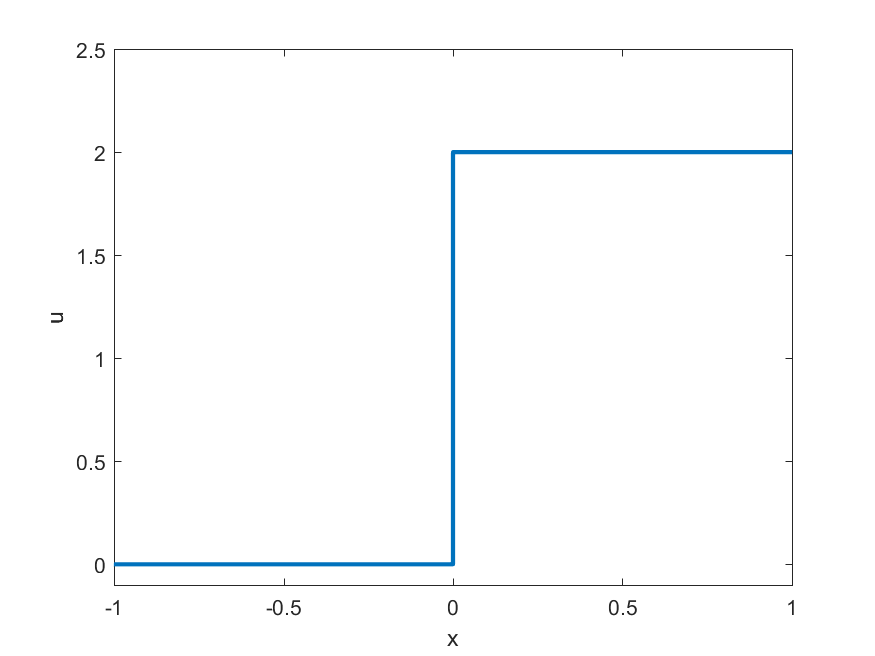}
\includegraphics[width=40mm,height=30mm]{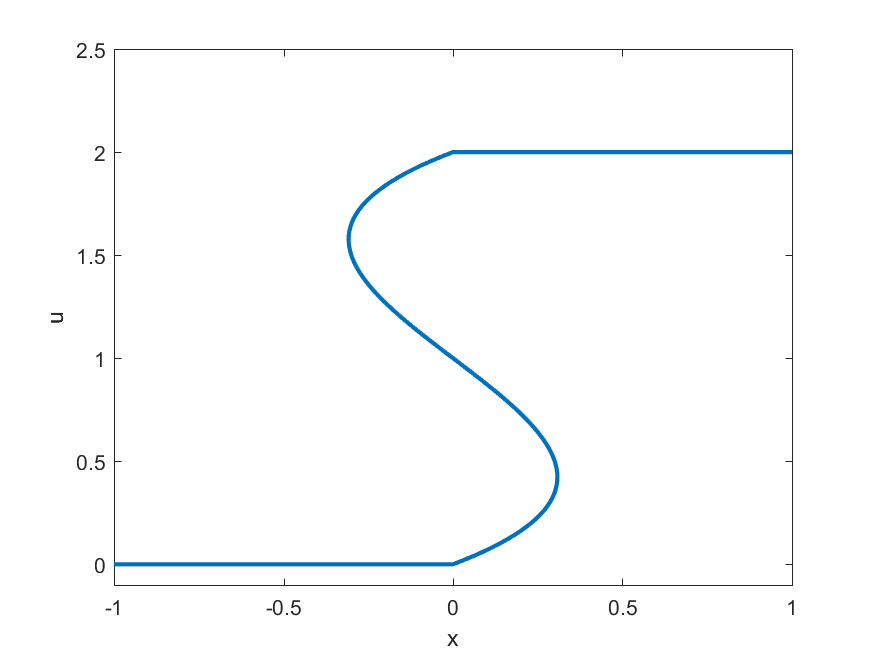}
\includegraphics[width=40mm,height=30mm]{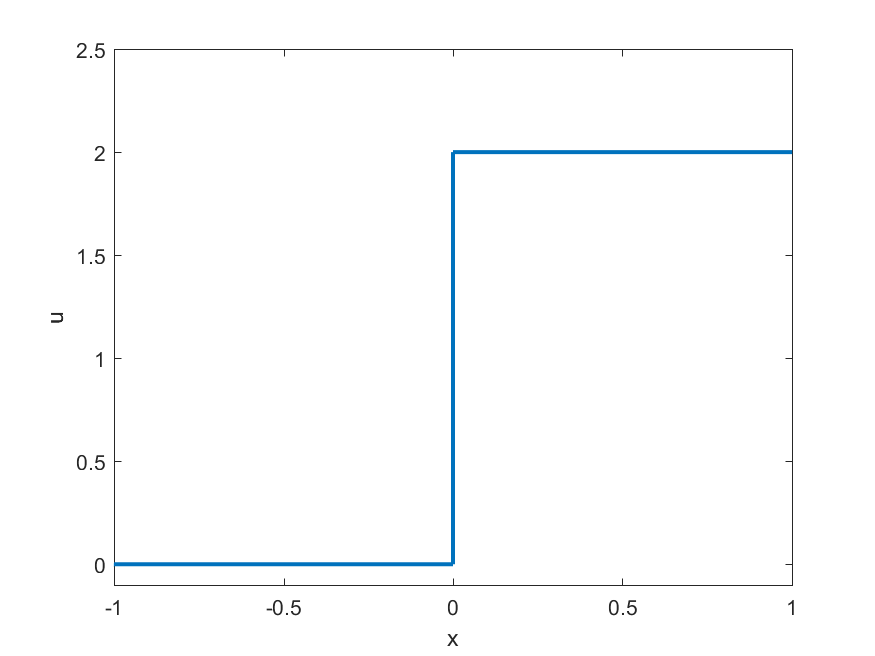}
\end{center}
\caption{Initial condition (left), characteristic flow (centre) and the Generalized Equal-Area Projection (right) associated with Example \ref{Example2}.}
\label{Example2Plots}
\end{figure}

\begin{figure}[!ht]
\begin{center}
\includegraphics[width=60mm,height=40mm]{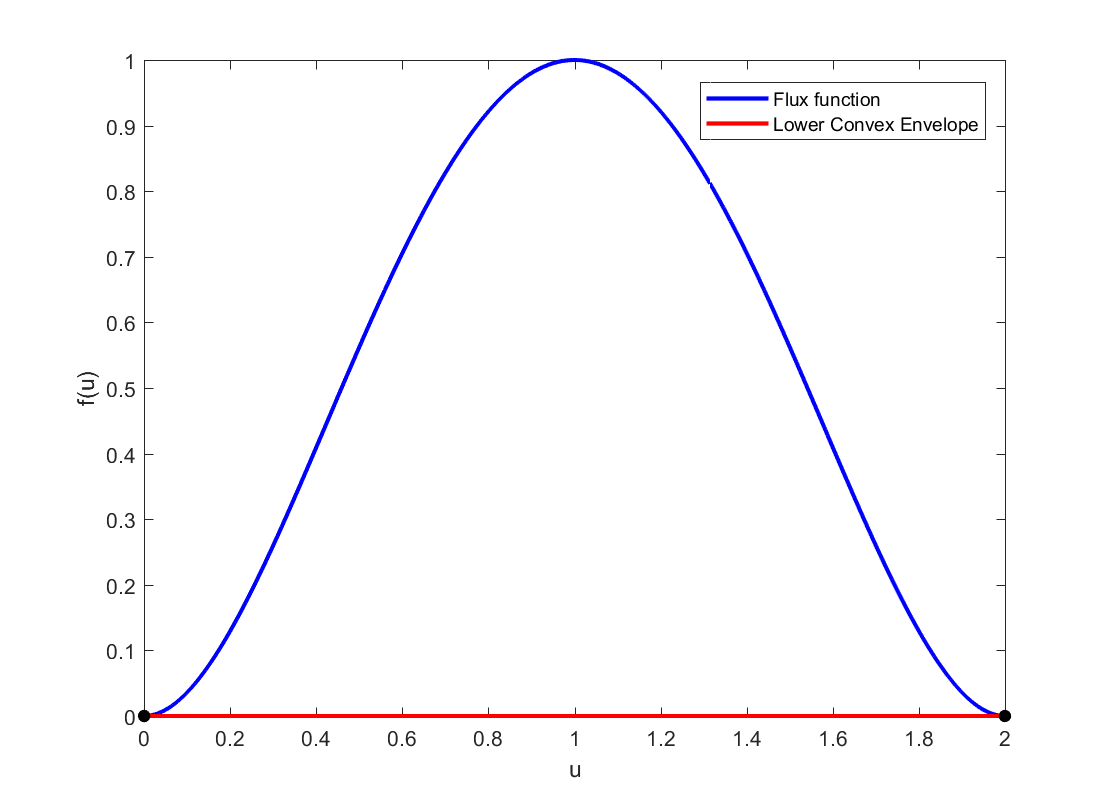}
\end{center}
\caption{Convex envelope corresponding to Example \ref{Example2}}
\label{Example2Envelope}
\end{figure}

\end{example}

\begin{example}\label{Example3}
\begin{equation}
\begin{cases}
u_t+(\frac{1}{4}u^4-\frac{5}{3}u^3+3u^2)_x=0\label{Example3Equation}\\
u(x,0)=
\begin{cases}
0 \quad x<0\\
3.5 \quad x\geq 0
\end{cases},
\end{cases}
\end{equation}
In this example, as seen in Figure \ref{Example3Plots}, our method predicts two rarefactions with a shock between them. This indeed corresponds to the lower convex hull as shown in Figure \ref{Example3Envelope}.

\begin{figure}[!ht]
\begin{center}
\includegraphics[width=40mm,height=30mm]{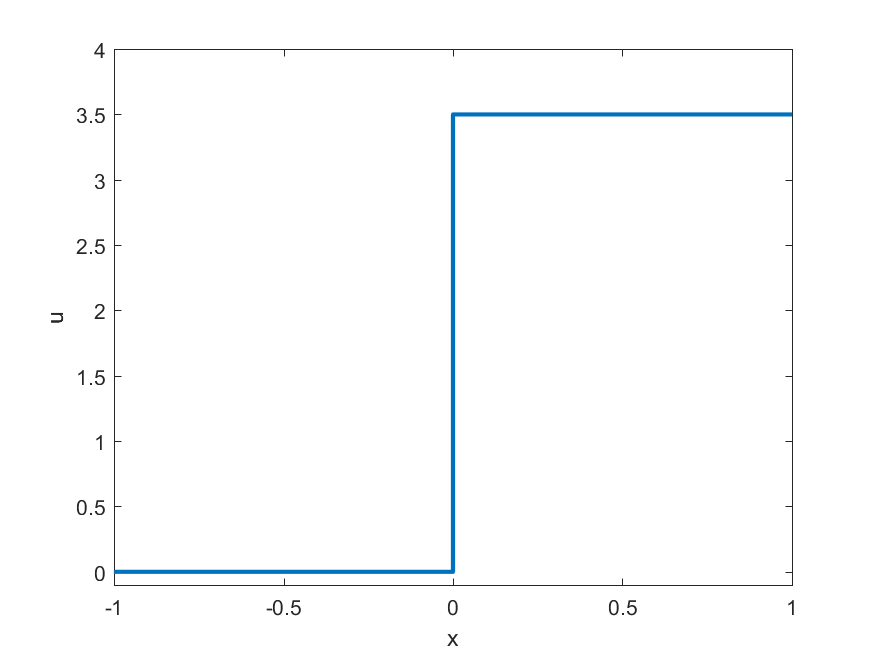}
\includegraphics[width=40mm,height=30mm]{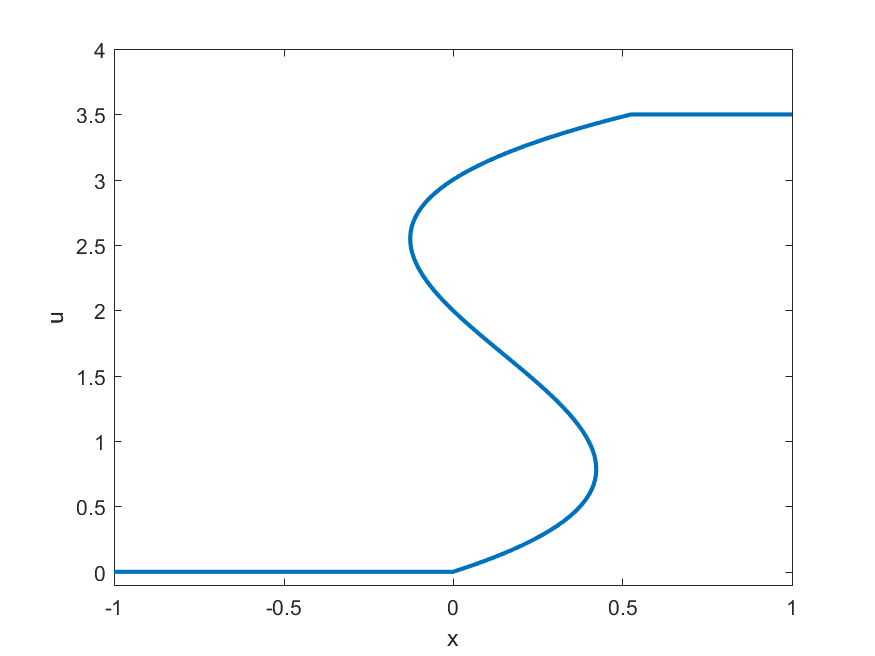}
\includegraphics[width=40mm,height=30mm]{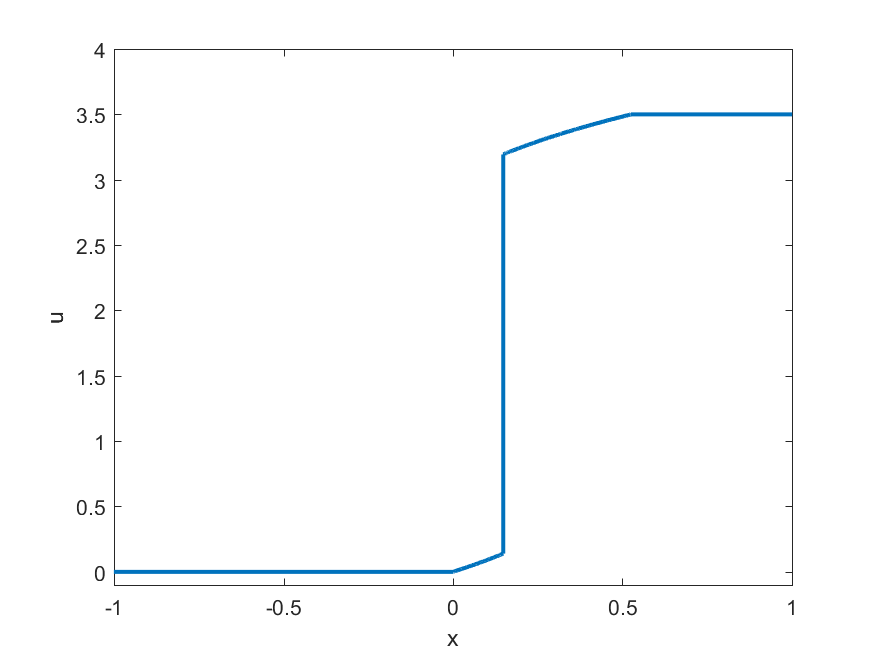}
\end{center}
\caption{Initial condition (left), characteristic flow (centre) and the Generalized Equal-Area Projection (right) associated with Example \ref{Example3}.}
\label{Example3Plots}
\end{figure}

\begin{figure}[!ht]
\begin{center}
\includegraphics[width=60mm,height=40mm]{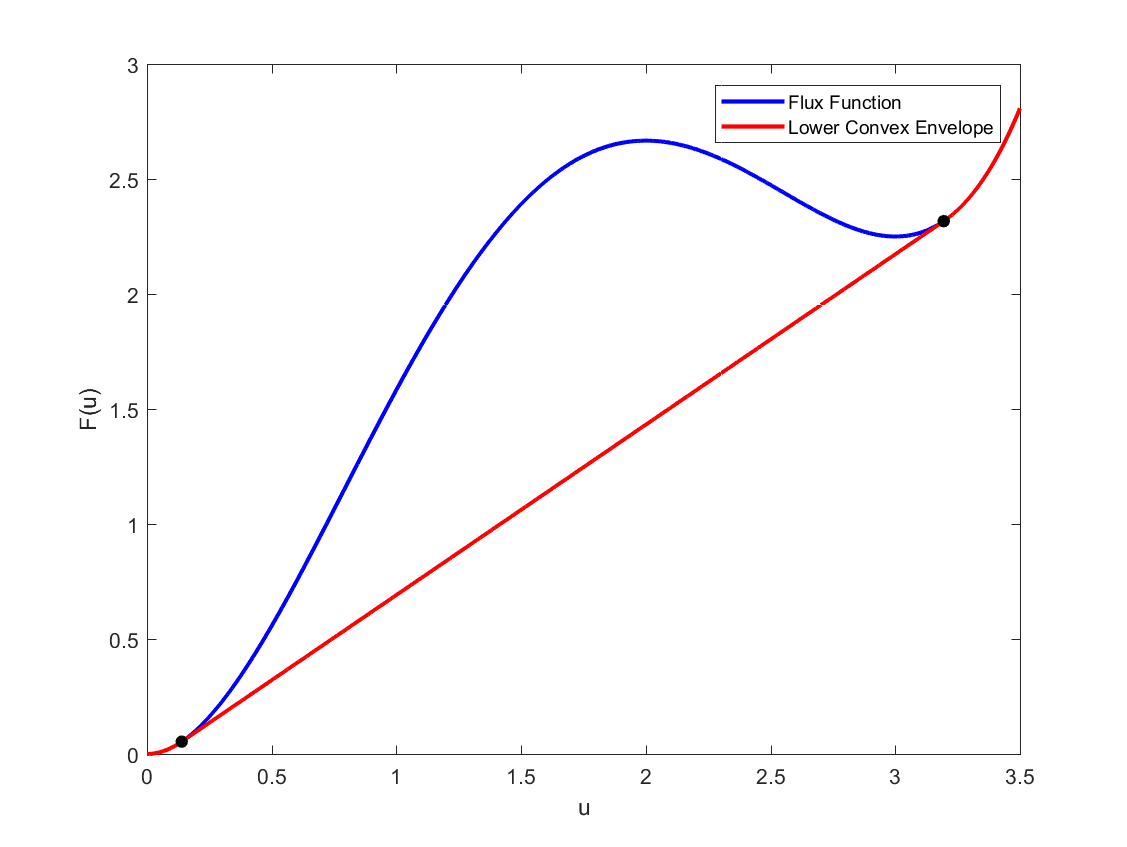}
\includegraphics[width=60mm,height=40mm]{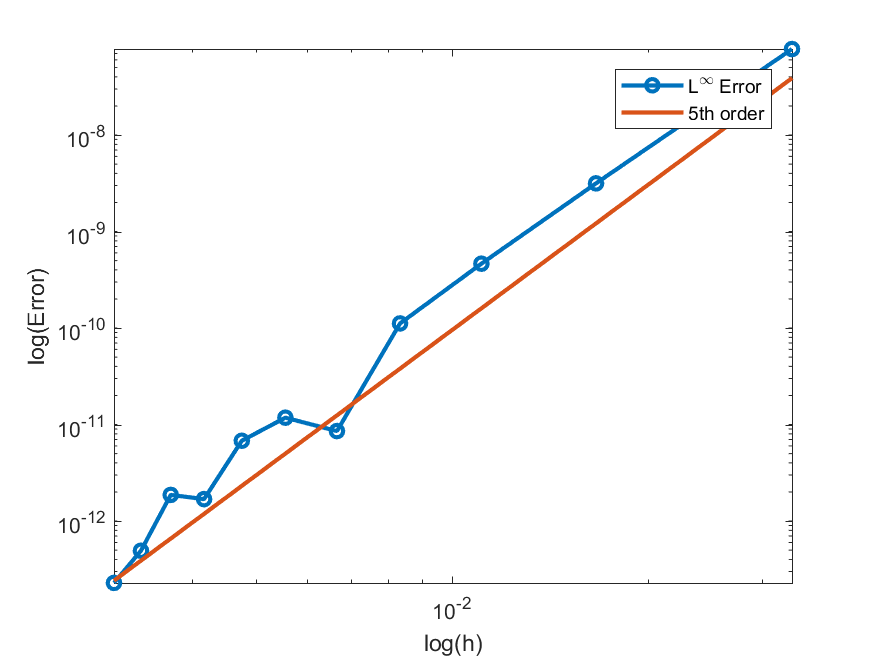}
\end{center}
\caption{Convex envelope (left) and convergence (right) corresponding to Example \ref{Example3}}
\label{Example3Envelope}
\end{figure}

\end{example}

The following example is the Buckley-Leverett equation for two-phase flow discussed in \cite{Buckley}.
\begin{example}\label{Example4}
\begin{equation}
\begin{cases}
u_t+(\frac{u^2}{u^2+\frac{1}{2}(1-u)^2})_x=0\label{Example4Equation}\\
u(x,0)=
\begin{cases}
1 \quad x<0\\
0 \quad x\geq 0
\end{cases},
\end{cases}
\end{equation}
In this example, as seen in Figure \ref{Example3Plots}, our method correctly predicts  a single rarefactions which connects to a shock. This indeed corresponds to the upper convex hull as shown in Figure \ref{Example3Envelope}. The convergence plot for this problem is quite noisy, although clearly we still obtain the desired fifth order accuracy. This is due to the rapidly growing derivatives of the flux function and the fact that we unable to take very small time steps without running into round-off error. This convergence plot was obtained by starting with 10 interpolants and doubling at each iteration. 

\begin{figure}[!ht]
\begin{center}
\includegraphics[width=40mm,height=30mm]{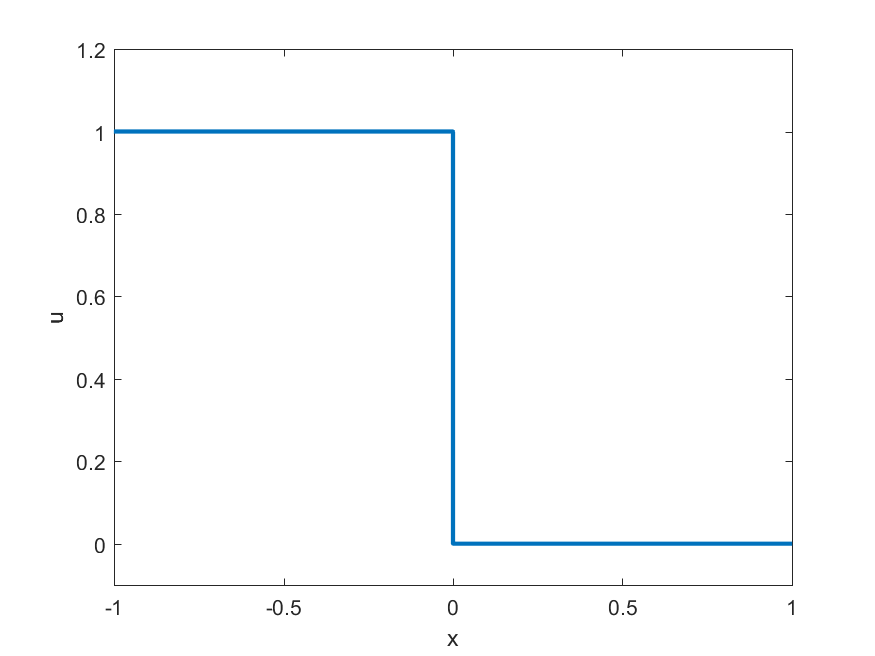}
\includegraphics[width=40mm,height=30mm]{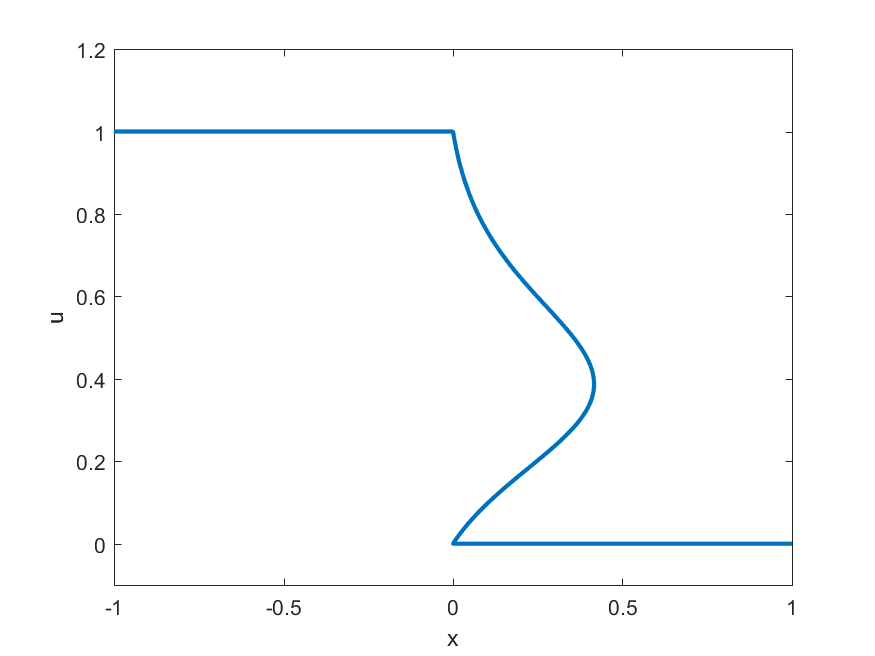}
\includegraphics[width=40mm,height=30mm]{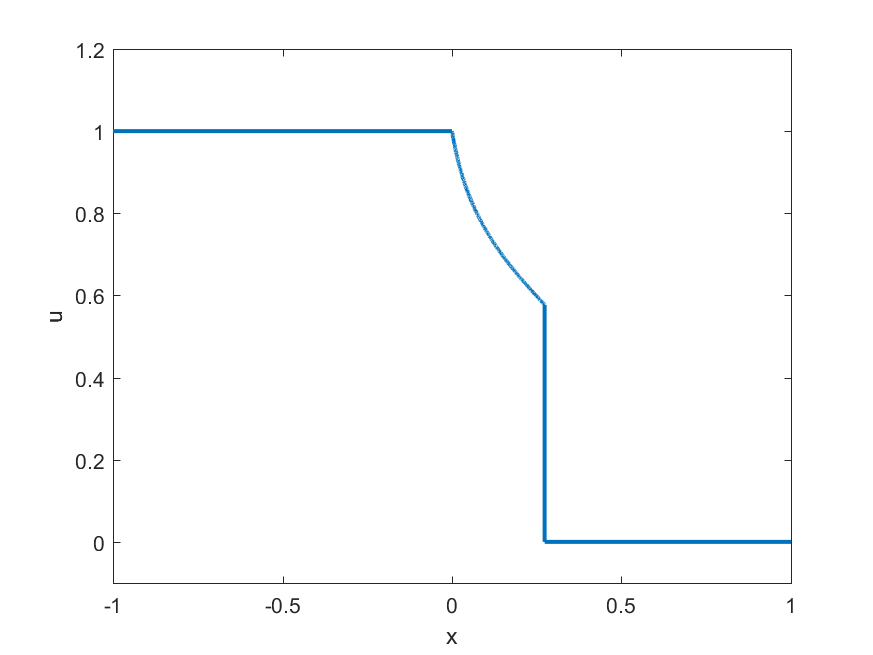}
\end{center}
\caption{Initial condition (left), characteristic flow (centre) and the Generalized Equal-Area Projection (left) associated with Example \ref{Example4}.}
\label{Example3Plots}
\end{figure}

\begin{figure}[!ht]
\begin{center}
\includegraphics[width=60mm,height=40mm]{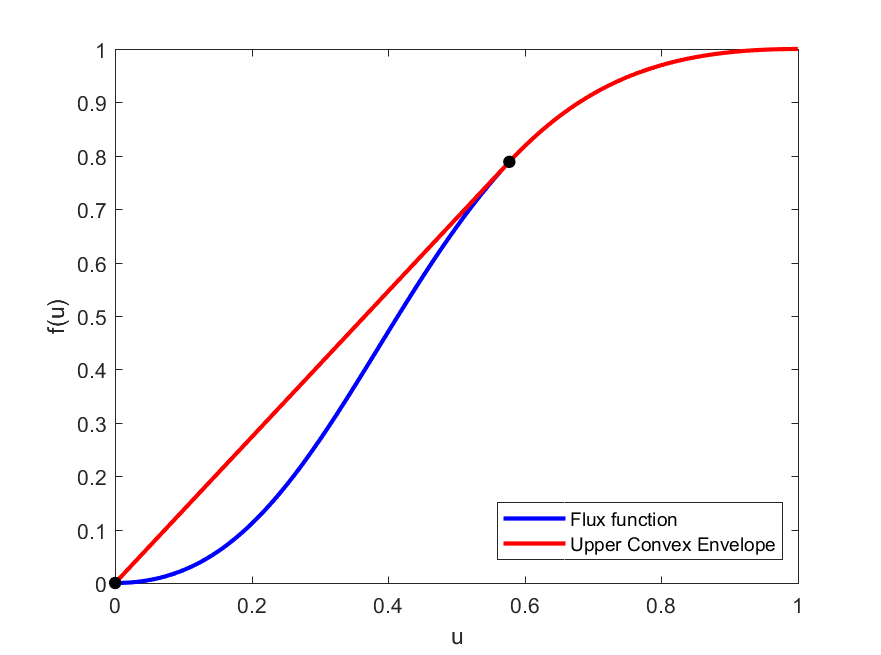}
\includegraphics[width=60mm,height=40mm]{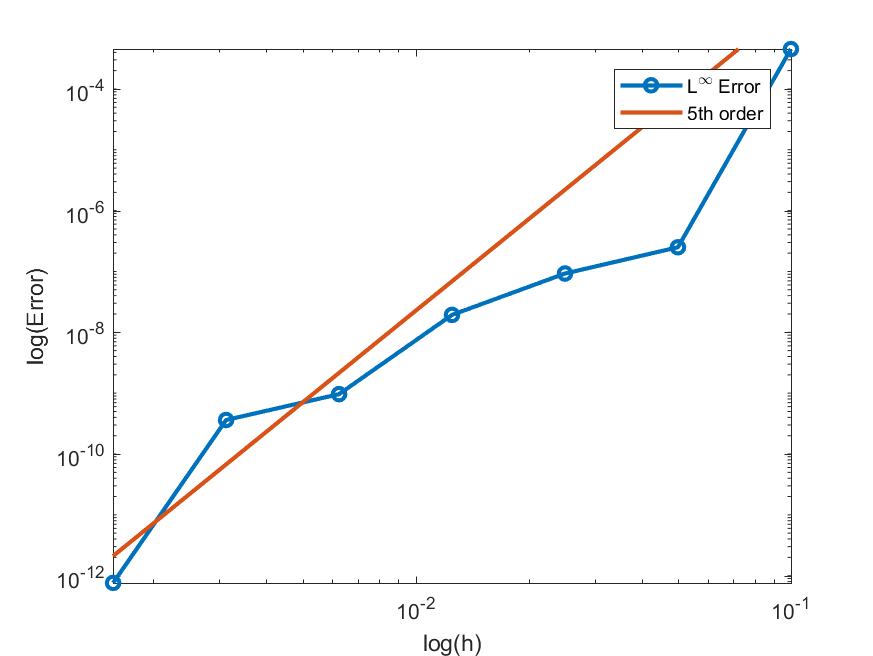}
\end{center}
\caption{Convex envelope (left) and convergence (right) corresponding to Example \ref{Example4}}
\label{Example3Envelope}
\end{figure}

\end{example}

For the next example we return to the same flux function as in Example \ref{Example3} however we begin with a different initial condition.

\begin{example}\label{Example5}
\begin{equation}
\begin{cases}
u_t+(\frac{1}{4}u^4-\frac{5}{3}u^3+3u^2)_x=0\label{Example5Equation}\\
u(x,0)=
\begin{cases}
0 \quad x<0\\
5 \quad 0\leq x< 5\\
0 \quad x\geq 5
\end{cases},
\end{cases}
\end{equation}
Here we see a box initial condition, which means we should have the weak solution given by the lower convex hull for the left Riemann problem and the upper convex hull for the right Riemann problem. The Generalized Equal-Area Principle Algorithm, which is being applied using the reverse algorithm discussed in Remark \ref{Reverse}, is shown step by step in Figure \ref{Example5Plots}. The resulting weak solution indeed captures the lower convex hull for the left Riemann problem and the upper for the right Riemann problem. Putting this all together in Figure \ref{Example5Envelope} yields the convex hull of the graph of the flux function $F$.

\begin{figure}[!ht]
\begin{center}
\includegraphics[width=40mm,height=30mm]{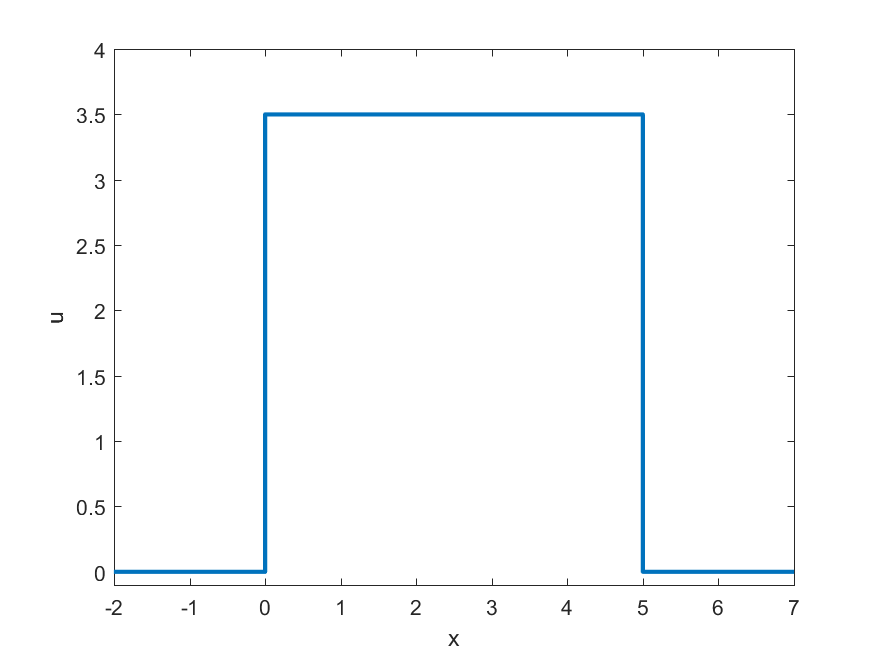}
\includegraphics[width=40mm,height=30mm]{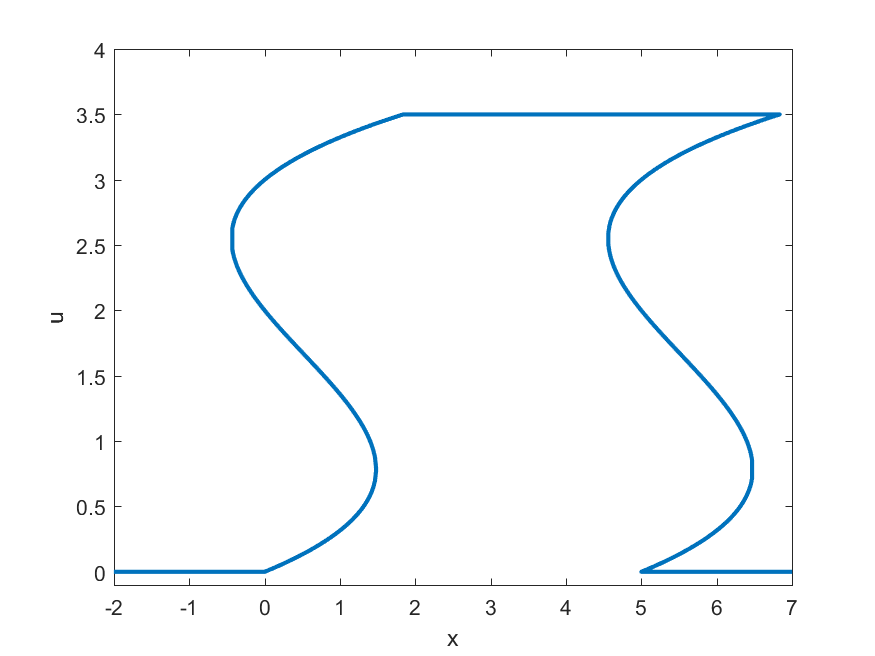}
\includegraphics[width=40mm,height=30mm]{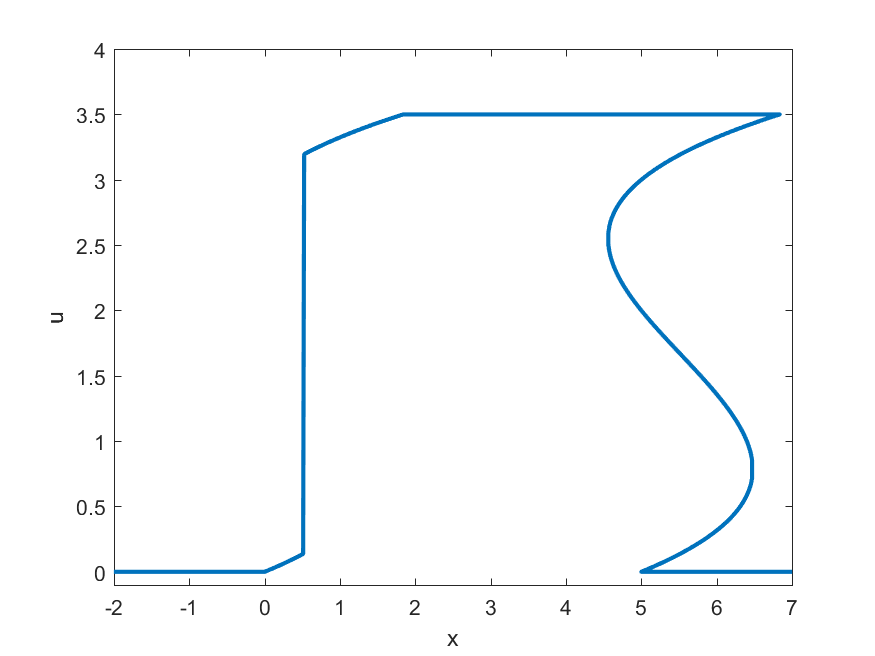}
\includegraphics[width=40mm,height=30mm]{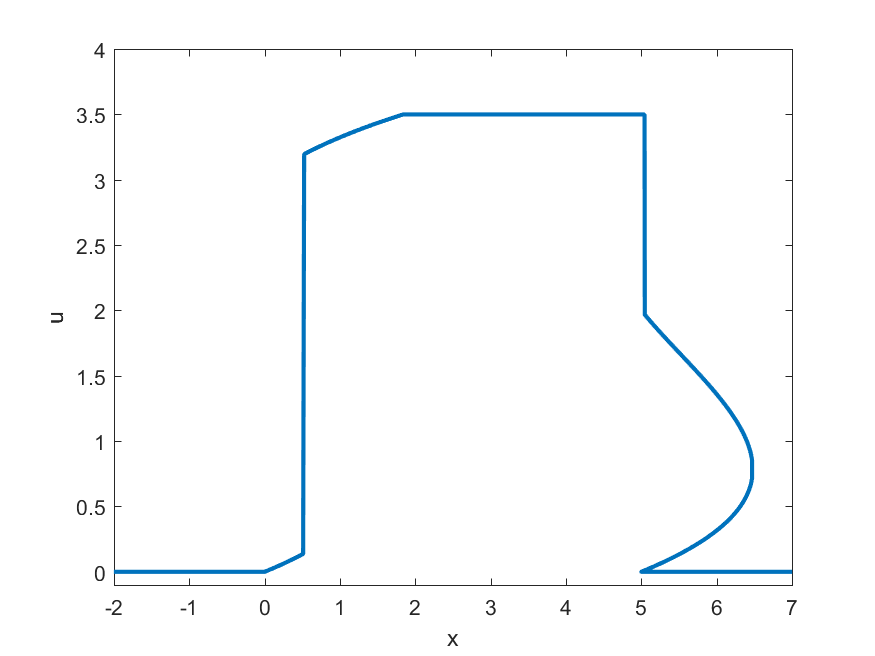}
\includegraphics[width=40mm,height=30mm]{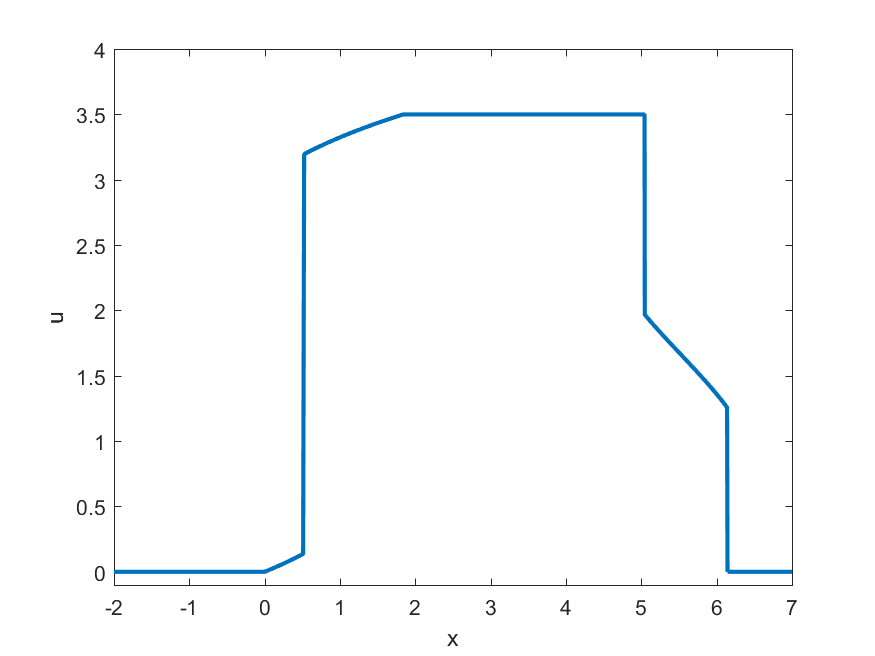}
\end{center}
\caption{Step by step application of the Generalized Equal-Area Principle for Example \ref{Example5}.}
\label{Example5Plots}
\end{figure}

\begin{figure}[!ht]
\begin{center}
\includegraphics[width=60mm,height=40mm]{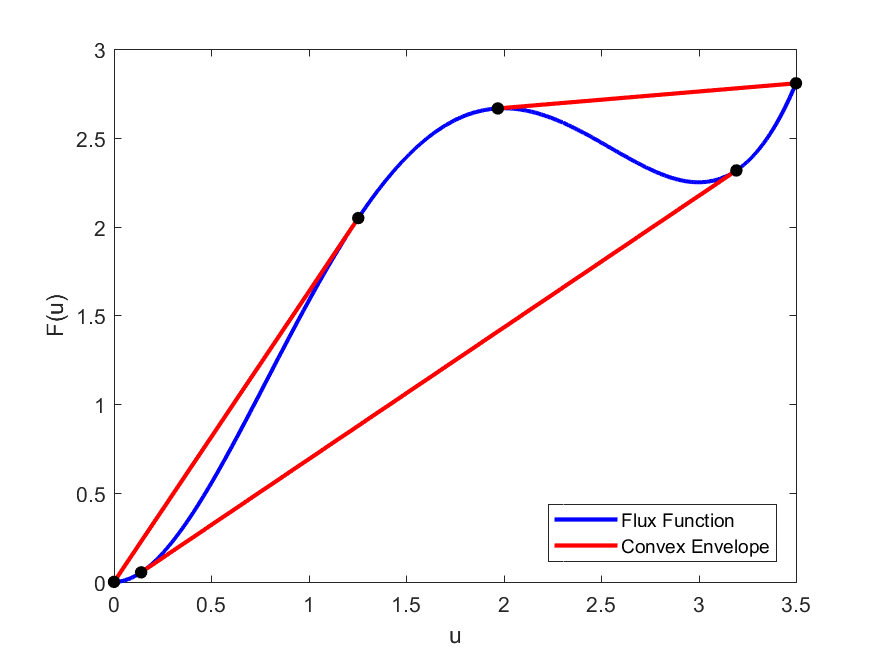}
\end{center}
\caption{Convex hull of the graph of the flux function given in Example \ref{Example5}.}
\label{Example5Envelope}
\end{figure}

\end{example}

\newpage
\section{Discussion}\label{Discussion} 
In this paper we set out to apply the area-preserving parametric interpolation framework of \cite{mcgregor2019area} to 1-D scalar conservation laws with non-convex flux functions. In Section \ref{PIF} we observed that when the parametric curve (\ref{ParCurve}) becomes multi-valued, a generic equal-area projection is insufficient, as uniqueness of equal-area solutions is lost in the non-convex case, as shown in Figure \ref{NonConRiemann2}. Using the fact that the convex envelope of the flux function can be used to construct weak solutions which satisfy the Oleinik entropy conditions \cite{Oleinik}, we describe a basic algorithm for constructing the upper and lower convex hulls. Leveraging the connection between secant lines of the convex hull of the flux function and equal-area curves of (\ref{ParCurve}), we derived the generalized equal-area principle algorithm. In Section \ref{Numerics} we showed that this algorithm indeed selects the appropriate upper or lower convex envelope (depending on the states $u_L$ and $u_R$) automatically. Additionally, when interpolating (\ref{ParCurve}) with the area-preserving interpolation of \cite{mcgregor2019area}, we indeed achieve fifth order accuracy in shock position.

The method presented here is high order, exactly conservative, flexible and selects weak solutions satisfying the Oleinik entropy criteria. Although this work focused on the Riemann problem, the framework is applicable to general piecewise-smooth initial data as well.  With smooth initial data, shocks form when the curve overturns and becomes multi-valued. In this situation the generalized equal-area principle is applied in the same way, therefore guaranteeing the entropy conditions are satisfied. At its core the method is very simple: Flow particles under the characteristic flow, interpolate the data using high order parametric interpolation, if the curve overturns, perform the correct projection in accordance with the generalized equal-area principle and then continue. This approach is valid for lower order interpolation as well, if a simple implementation is desired, however the conservative nature and high order accuracy of the area-preserving parametric interpolation is ideal.   For these reasons we believe this method can be a valuable tool for simulation of 1-D scalar conservation laws with non-convex flux functions, such as traffic flow or Buckley-Leverett type equations for two-phase flow. Looking forward we are interested in exploring the non-homogeneous case with non-convex flux functions. As discussed in \cite{mcgregor2019parametric}, we know the equal-area principle fails to yield the correct weak solution in the convex case, therefore a modification to the methods discussed here will be required to obtain a high order methods for non-convex problems with source terms.
\bibliographystyle{abbrv}
\bibliography{references}

\end{document}